\documentclass[10pt]{amsart}
\usepackage{hyperref}
\usepackage{amsmath}
\usepackage{amsthm}
\usepackage{multirow}
\usepackage{mathrsfs}
\usepackage{amssymb}
\usepackage{mathtools}
\usepackage{xcolor}
\usepackage{algorithm}
\usepackage{algorithmic}
\usepackage{cite,url}
\usepackage{latexsym}
\usepackage{enumitem}
\usepackage{booktabs}
\usepackage{bm}
\usepackage{graphicx}
\usepackage{textcomp}
\usepackage{caption}
\usepackage{subfigure}
\usepackage[toc,page]{appendix}
\newcommand{\Rni}{\mathbb{R}^{n_{i}}}

\newcommand{\nfR}{\mathbb{R}}
\newcommand{\xpi}{x_i}

\newcommand{\xmi}{x_{-i}}
\newcommand{\umi}{u_{-i}}
\newcommand{\mbF}{\mbox{F}}

\newcommand{\fpi}{f_{i}}

\newcommand{\lip}{\left<}
\newcommand{\rip}{\right>}
\newcommand{\lvt}{\left[}
\newcommand{\rvt}{\right]}
\newcommand{\qmod}{\mbox{Qmod}}
\newcommand{\idl}{\mbox{Ideal}}
\newcommand{\ddd}{,\dots,}

\DeclareMathOperator{\Rank}{rank}

\newcommand{\re}{\mathbb{R}}
\newcommand{\cpx}{\mathbb{C}}

\newcommand{\N}{\mathbb{N}}

\def\af{\alpha}

\newcommand{\st}{\mathit{s.t.}}
\newcommand{\reff}[1]{(\ref{#1})}

\newcommand{\lmd}{\lambda}

\newcommand{\dt}{\delta}

\newcommand{\mc}[1]{\mathcal{#1}}

\def\rank{\mbox{rank}}
\def\cod{\mbox{codim}}
\newcommand{\bdes}{\begin{description}}
	\newcommand{\edes}{\end{description}}

\newcommand{\bal}{\begin{align}}
	\newcommand{\eal}{\end{align}}

\newcommand{\bnum}{\begin{enumerate}}
	\newcommand{\enum}{\end{enumerate}}

\newcommand{\bit}{\begin{itemize}}
	\newcommand{\eit}{\end{itemize}}

\newcommand{\bea}{\begin{eqnarray}}
	\newcommand{\eea}{\end{eqnarray}}
\newcommand{\be}{\begin{equation}}
	\newcommand{\ee}{\end{equation}}

\newcommand{\baray}{\begin{array}}
	\newcommand{\earay}{\end{array}}

\newcommand{\bsry}{\begin{subarray}}
	\newcommand{\esry}{\end{subarray}}

\newcommand{\bca}{\begin{cases}}
	\newcommand{\eca}{\end{cases}}

\newcommand{\bcen}{\begin{center}}
	\newcommand{\ecen}{\end{center}}

\newcommand{\bbm}{\begin{bmatrix}}
	\newcommand{\ebm}{\end{bmatrix}}

\newcommand{\btab}{\begin{tabular}}
	\newcommand{\etab}{\end{tabular}}

\theoremstyle{definition}
\newtheorem{definition}{Definition}[section]

\newtheorem{lemma}[definition]{Lemma}
\newtheorem{theorem}[definition]{Theorem}

\newtheorem{example}[definition]{Example}
\newtheorem{alg}[definition]{Algorithm}
\theoremstyle{remark}
\newtheorem{remark}[definition]{Remark}

\numberwithin{equation}{section}

\author{Kisun Lee}
\address{Department of Mathematics, University of California San Diego, 9500 Gilman Drive, La Jolla, CA 92093}
\email {kil004@ucsd.edu}\urladdr{https://klee669.github.io}

\author{Xindong Tang}
\address{Department of Applied Mathematics, The Hong Kong Polytechnic University, Hung Hom, Kowloon, Hong Kong}
\email {xindong.tang@polyu.edu.hk}\urladdr{https://www.polyu.edu.hk/ama/profile/txd/}

\date{\today}

\begin{document}

	\title[Polyhedral homotopy for GNEPs of polynomials]{On the polyhedral homotopy method for solving generalized Nash equilibrium problems of polynomials}
	
	\date{}
	
	\begin{abstract}
		The \emph{generalized Nash equilibrium problem} (GNEP) is a kind of game to find strategies
		for a group of players such that each player's objective function is optimized.
		Solutions for GNEPs are called \emph{generalized Nash equilibria} (GNEs).
		In this paper, we propose a numerical method for finding GNEs of GNEPs of polynomials based on the polyhedral homotopy continuation and the Moment-SOS hierarchy of semidefinite relaxations.
		We show that our method can find all GNEs if they exist,
		or detect the nonexistence of GNEs,
		under some genericity assumptions.
		Some numerical experiments are made to demonstrate the efficiency of our method.
	\end{abstract}
	\maketitle
	\makeatletter
	\newcommand{\subjclass}[2][1991]{%
		\let\@oldtitle\@title%
		\gdef\@title{\@oldtitle\footnotetext{#1 \emph{Mathematics subject classification.} #2}}%
	}
	\newcommand{\keywords}[1]{%
		\let\@@oldtitle\@title%
		\gdef\@title{\@@oldtitle\footnotetext{\emph{Key words and phrases.} #1.}}%
	}
	\makeatother
	\section{Introduction}
	
	Suppose there are  $N$  players and the $i$th player's strategy is a vector
	$x_{i}\in\mathbb{R}^{n_{i}}$ (the $n_i$-dimensional real Euclidean space).
	We write that
	\[x_i:=(x_{i,1},\ldots,x_{i,n_i}),\quad x:=(x_1,\ldots,x_N).\]
	The total dimension of all strategies is $n := n_1+ \cdots + n_N.$
	When the $i$th player's strategy is considered,
	we use $x_{-i}$ to denote the subvector of all players' strategies
	except the $i$th one, i.e., 
	\[
	x_{-i} \, := \, (x_1, \ldots, x_{i-1}, x_{i+1}, \ldots, x_N),
	\]
	and write $x=(x_{i},x_{-i})$ accordingly.
	The \emph{generalized Nash equilibrium problem} (GNEP) is to find a tuple of strategies
	$u = (u_1, \ldots, u_N)$
	such that each $u_i$ is a minimizer of the $i$th player's optimization
	\be
	\label{eq:GNEP}
	\mbox{F}_i(u_{-i}):
	\left\{ \begin{array}{cl}
		\min\limits_{\xpi\in \Rni}  &  \fpi(x_i,u_{-i}) \\
		\st & g_{i,j}(x_i,u_{-i})  = 0, (j\in\mc{E}_i), \\
		& g_{i,j}(x_i,u_{-i})  \ge 0, (j\in\mc{I}_i).
	\end{array} \right.
	\ee
	In the above, the $\mc{E}_i$ and $\mc{I}_i$ are disjoint labeling sets (possibly empty),
	the $f_i$ and $g_{i,j}$ are continuous functions in $x$
	and we suppose $\mc{E}_i\cup \mc{I}_i=\{1,\dots,m_i\}$ for each $i\in\{1,\dots,N\}$.
	A solution to the GNEP is called a \emph{generalized Nash equilibrium} (GNE).
	If defining functions $f_i$ and $g_{i,j}$ are polynomials in $x$ for all $i\in\{1,\dots,N\}$ and $j\in\{1,\dots,m_i\}$,
	then we call the GNEP a \emph{generalized Nash equilibrium problem of polynomials}.
	Besides that, we let $X_i$ be the point-to-set map such that
	\begin{equation*}
		X_i(\xmi) \, := \,
		\left\{x_i \in \Rni \left|\begin{array}{l}
			g_{i,j}(\xpi,\xmi) = 0 ,\, (j\in \mathcal{E}_i),\\
			g_{i,j}(\xpi,\xmi) \geq 0 ,\, (j\in \mathcal{I}_i)
		\end{array}\right.\right\}.
	\end{equation*} 
	Then, $X_i(\xmi)$ is the $i$th player's feasible set for the given other players' strategies $\xmi$.
	Let \[X:=\{x\in\re^n\mid x_i\in X_i(x_{-i}) \mbox{ for all } i=1\ddd N\}.\]
	We say $x\in\re^n$ is \emph{feasible} for this GNEP if $x\in X$.
	Moreover, if for every $i\in\{1,\dots,N\}$ and $j\in\{1,\dots,m_i\}$, the constraining function $g_{i,j}$ only has the variable $x_i$,
	i.e.,  the $i$th player's feasible strategy set is independent of other players' strategies,
	then the GNEP is called a \emph{(standard) Nash equilibrium problem} (NEP).
	The GNEP is said to be {\it convex} if for each $i$ and all $\xmi$ such that $X_i(\xmi)\ne\emptyset$, the $\mbF_i(\xmi)$ is a convex optimization,
	i.e.,  $f_i(\xpi,\xmi)$ is convex in $\xpi$,
	$g_{i,j}\,(j\in\mc{E}_i)$ is linear in $\xpi$,
	and $g_{i,j}\,(j\in\mc{I}_i)$ is convex in $\xpi$.

	GNEPs originated from economics in \cite{debreu1952social,arrow1954existence},
	and have been widely used in many other areas, such as telecommunications \cite{
		ardagna2015generalized}, supply chain \cite{pang2008distributed} and machine learning \cite{liu2016gan}.
	There are plenty of interesting models formulated as GNEP of polynomials,
	and we refer to \cite{pang2008distributed,Nie2021convex,Nie2020nash,facchinei2010penalty,facchinei2010generalized} for them.
	
	For recent studies on GNEPs, one primary task is to develop efficient methods for finding GNEs.
	Indeed, solving GNEPs may easily be out of reach, especially when convexity assumptions are not given.
	For NEPs, some methods are studied in \cite{gurkan2009approximations,ratliff2013characterization}.
	When the NEP is defined by polynomials,
	a method using the Moment-SOS semidefinite relaxation on the KKT system is introduced in \cite{Nie2020nash}.
	For GNEPs, people mainly consider solution methods under convexity assumptions,
	such as the penalty method \cite{facchinei2010penalty,ba2020exact},
	the augmented Lagrangian method \cite{kanzow2016augmented}, 
	the variational and quasi-variational inequality approach \cite{Facchinei2007generalized,facchinei201012,harker1991generalized},
	the Nikaido-Isoda function approach \cite{von2009optimization,von2009relaxation},
	and the interior point method on solving the KKT system \cite{dreves2011solution}.
	Moreover, for convex GNEP of polynomials,
	a semidefinite relaxation method is introduced in \cite{Nie2021convex},
	and it is extended to nonconvex rational GNEPs in \cite{nie2021rational}.
	The Gauss-Seidel method is studied in \cite{nie2021gauss} for nonconvex GNEPs of polynomials.
	We refer to \cite{facchinei2010generalized,facchinei201012} for surveys on GNEPs.
	
	In this paper, we study GNEPs of polynomials.
	The problems without convexity assumptions are mainly considered.
	We propose a method for finding GNEs based on the polyhedral homotopy continuation and the Moment-SOS semidefinite relaxations, and investigate its properties.
	Our main contributions are:
	\begin{itemize}
		\item We propose a numerical algorithm for solving GNEPs of polynomials.
		The polyhedral homotopy continuation is exploited for solving the complex KKT system of GNEPs,
		and we select GNEs from the set of complex KKT points with the help of Moment-SOS semidefinite relaxations.
		
		\item We show that when the GNEP is given by dense polynomials whose coefficients are generic,
		the mixed volume for the complex KKT system is identical to its algebraic degree.
		In this case, the polyhedral homotopy continuation can obtain all complex KKT points, and our algorithm finds all GNEs if they exist, or detect their nonexistence of them.
		
		\item Even when the number of complex KKT points obtained by the polyhedral homotopy is less than the mixed volume, or there exist infinitely many complex KKT points, our algorithm may still find one or more GNEs.
		
		\item Numerical experiments are presented to show the effectiveness of our algorithm.
	\end{itemize}
	This paper is organized as follows.
	In Section~\ref{sc:pre},
	we introduce some basics in optimality conditions for GNEPs, polyhedral homotopy and polynomial optimization.
	The algorithm for solving GNEPs of polynomials is proposed in Section~\ref{sc:alg}.
	We show the polyhedral homotopy continuation is optimal for GNEPs of polynomials with generic coefficients in Section~\ref{sc:mv}.
	Numerical experiments are presented in Section~\ref{sc:ne}.

	\section{Preliminaries}
	\label{sc:pre}
	
	In this section, preliminary concepts for GNEPs, polyhedral homotopy continuation and polynomial optimization problems are reviewed. We introduce the optimality conditions for GNEPs to derive a system of polynomials whose collection of solutions contains GNEs. Then, we review Bernstein's theorem, which gives an upper bound of the number of solutions for a system of polynomials. Finally, the polyhedral homotopy continuation and the Moment-SOS hierarchy of semidefinite relaxations are suggested as two main tools to solve GNEPs.

	\subsection{Optimality conditions for GNEPs}
	
	Under some suitable constraint qualifications (for example, the linear constraint qualification condition (LICQ), or Slater's Condition for convex problems; see \cite{Brks}),
	if $u\in X$ is a GNE, then there exist Lagrange multiplier vectors $\lambda_1,\ldots,\lambda_N$ such that
	\be
	\label{eq:KKTsystem}
	\left\{
	\begin{array}{ll}
		\nabla_{x_i} f_i(x)-\sum\limits_{j=1}^{m_i}\lambda_{i,j}\nabla_{x_i} g_{i,j}(x)=0, \, &(i\in\{1,\dots,N\})\\
		\lambda_i\perp g_i(x),\,g_{i,j}(x)=0,\,&(i\in\{1,\dots,N\},j\in\mathcal{E}_i)\\
		\lambda_{i,j}\ge0,\,g_{i,j}(x)\ge0,\,&(i\in\{1,\dots,N\},j\in\mathcal{I}_i)
	\end{array}
	\right.
	\ee
	where $\nabla_{x_i}f_i(x)$ is the gradient of $f_i$ with respect to $x_i$ and $\lambda_i\perp g_i(x)$ implies {that $\lambda_{i}^\top g_{i}(x)=0$ for the constraining vector
		\[ g_i(x)\,:=\,\left[\ g_{i,1}(x),\,\dots,\,g_{i,m_i}(x)\ \right]^\top.\]    
		The polynomial system  (\ref{eq:KKTsystem}) is called the \emph{KKT system} of this GNEP. 
		The solution $(x,\lambda_1\ddd\lambda_N)$ of the KKT system is called a \emph{KKT tuple}
		and the first block of coordinates $x$ is called a \emph{KKT point}.
		For the GNEP of polynomials,
		the LICQ of $\mbox{F}_i(u_{-i})$ hold at every GNE \cite{Nie2022degree}, under some genericity conditions.
		Moreover, consider the system consists of all equations in (\ref{eq:KKTsystem}),
		i.e., 
		\be
		\label{eq:KKTep}
		\left\{
		\begin{array}{ll}
			\nabla_{x_i} f_i(x)-\sum\limits_{j=1}^{m_i}\lambda_{i,j}\nabla_{x_i} g_{i,j}(x)=0, \, &(i\in\{1,\dots,N\})\\
			\lambda_i\perp g_i(x),\,g_{i,j}(x)=0,\,&(i\in\{1,\dots,N\},j\in\mathcal{E}_i).
		\end{array}
		\right.
		\ee
		Then, the $(x,\lambda_1\ddd\lambda_N)$ satisfying (\ref{eq:KKTep}) is called a \emph{complex KKT tuple},
		and similarly, the first coordinate $x$ is called a \emph{complex KKT point}.
		For generic GNEPs of polynomials,
		there are finitely many solutions to (\ref{eq:KKTep}) \cite{Nie2022degree}.
		In this case,
		the number of complex KKT points is called the \emph{algebraic degree of the GNEP}.
		
		\subsection{Mixed volumes and Bernstein's theorem}\label{subsection:mixedVolume}
		
		Let $\mathbb{C}[z_1,\dots,z_k]$ be the set of all complex coefficient polynomials in variables $z_1\ddd z_k$. 
		For a polynomial $p\in\mathbb{C}[z_1,\dots, z_k]$, suppose
		\[p=\sum_{a\in \N^k} c_a z^a\]
		where $z^a = z_1^{a_1}\cdots z_k^{a_k}$ for $a=(a_1,\dots, a_k)$.
		Then the \emph{support} of $p$, denoted by $A_p$,
		is the set of exponent vectors for monomials such that
		\[a\in A_p\quad \mbox{if and only if} \quad c_a\ne0.\]
		The convex hull \emph{$Q_p$} of the support $A_p$ is called the \emph{Newton polytope} of $p$. 
		For an integer vector $w\in \mathbb{Z}^k$,
		we define a map $h_w:\mathbb{Z}^k\to\mathbb{Z}$ such that
		\[h_w(a)=\langle w,a\rangle \quad \mbox{for all}\quad  a\in \mathbb{Z}^k.\]
		Given a finite integer lattice of points $A\subset \mathbb{Z}^k$,
		the minimum value of $h_w$ on $A$ is denoted by \emph{$h_w^\star(A)$}.
		When it is clear from the context,
		we may omit the index $w$ for $h$.
		Moreover, we let \emph{$A^w$}$:=\{a\in A\mid  h_w(a)=h_w^\star(A)\}$.
		Then,
		we define \emph{$p^w$} be the polynomial consists of terms of $p$ supported on $A^w$,
		i.e., for each $p=\sum_{a\in A}c_az^a\in\cpx[z_1,\dots, z_k]$, we have
		\[p^w=\sum_{a\in A^w}c_az^a.\]
		For an $m$-tuple of polynomials $\mathscr{P}=(p_1,\dots, p_m)$, 
		we denote $\mathscr{P}^w:=(p_1^w,\dots, p_m^w)$. 
		The $m$-tuple $\mathscr{P}^w$ is called the \emph{facial system} of $\mathscr{P}$ with respect to $w$. 
		The term `facial' comes from the fact that $A^w$ is a face of $A$ exposed by a vector $w$.
		
		Let $Q_1,\dots, Q_m$ be polytopes in $\mathbb{R}^k$, and $\alpha_1,\dots, \alpha_m$ be nonnegative real scalars. The \emph{Minkowski sum} of polytopes is
		$$\alpha_1Q_1+\cdots+\alpha_mQ_m:=\{\alpha_1v_1+\cdots+\alpha_nv_m\mid v_i\in Q_i\}.$$
		The volume of the Minkowski sum $\alpha_1Q_1+\cdots+\alpha_mQ_m$ can be understood as a homogeneous polynomial in variables of $\alpha_1,\dots,\alpha_m$.
		In particular, the coefficient for the term $\alpha_1\alpha_2\cdots \alpha_m$ in the volume of $\alpha_1Q_1+\cdots +\alpha_mQ_m$ is called the \emph{mixed volume} of $Q_1,\dots, Q_m$, which is denoted by \emph{$MV(Q_1,\dots, Q_m)$}.
		
		In \cite{Bernstein1975number}, it was proved that for a square polynomial system in $\cpx[z_1,\dots, z_k]$, the mixed volume of the system is an upper bound for the number of isolated roots in the complex torus $(\mathbb{C}\setminus \{\mathbf{0}\})^k$,
		where $\mathbf{0}$ is the all-zero vector.
		This is called \emph{Bernstein's theorem}.
		Moreover, 
		it states when the mixed volume bound is tight. 
		
		\begin{theorem}[Bernstein's theorem]\cite[Theorems A and B]{Bernstein1975number}\label{thm:BernsteinThm}
			Let $\mathscr{P}$ be a system consists of polynomials $p_1,\dots, p_k$ in $\mathbb{C}[z_1,\dots,z_k]$.  For each %support $A_{p_i}$ of $p_i$ and its 
			Newton polytope $Q_{p_i}$ of $p_i$, we have 
			\be\label{eq:berngeneric}
			(\text{the number of isolated roots for }\mathscr{P}\text{ in }(\mathbb{C}\setminus\{\boldsymbol{0}\})^k)\leq MV(Q_{p_1},\dots, Q_{p_k}).\ee
			The equality holds
			if and only if the facial system $\mathscr{P}^w$ has no root in $(\mathbb{C}\setminus\{\mathbf{0}\})^k$ for any nonzero $w\in \mathbb{Z}^k$.
		\end{theorem}
		It is worth noting that Bernstein's theorem concerns roots in the torus $(\mathbb{C}\setminus \{\mathbf{0}\})^k$ because it allows considering Laurent polynomials which are possible to have negative exponents. A system that satisfies the equality in the above theorem is called \emph{Bernstein generic}. 
		
		\subsection{Polyhedral homotopy continuation}\label{sc:prehom}
		The \emph{homotopy continuation} is an algorithmic method to find numerical approximations of roots for a system of polynomial equations. 
		Consider a square system of polynomial equations $\mathscr{P}\,\coloneqq\, \{p_1\ddd p_k\}\subset\mathbb{C}[z_1,\dots, z_k]$ with $k$ equations and $k$ variables. We are interested in solving the system $\mathscr{P}$, i.e., computing a zero set
		\[\mathbf{V}(\mathscr{P}):=\{z\in \mathbb{C}^k\mid p_1(z)=\cdots=p_k(z)=0\}.\]
		The main idea for the homotopy continuation is solving $\mathscr{P}$ by tracking a homotopy path from the known roots of a system $\mathscr{Q}$, called a \emph{start system}, to that of the \emph{target system} $\mathscr{P}$.
		Given the start system $\mathscr{Q}=\{q_1,\dots, q_k\}\subset\mathbb{C}[z_1,\dots, z_k]$ with the same number of variables and equations of $\mathscr{P}$, we construct a homotopy $\mathscr{H}(z,t)$ such that $\mathscr{H}(z,0)=\mathscr{Q}$ and $\mathscr{H}(z,1)=\mathscr{P}$. 
		For tracking the homotopy from $t=0$ to $t=1$, {numerical ODE solving techniques for \emph{Davidenko equations} and Newton's iteration are applied; see \cite[Chapter 2]{SommeseWampler:2005} for more details.}
		
		Choosing a proper start system is an important task for the homotopy continuation as it determines the number of paths to track. In this paper, the \emph{polyhedral homotopy continuation} established by Huber and Sturmfels \cite{huber1995polyhedral} is considered. For each polynomial $p_i$ in $\mathscr{P}$ with its Newton polytope $Q_{p_i}$, the polyhedral homotopy continuation constructs a start system $\mathscr{Q}$ with the mixed volume $MV(Q_{p_1},\dots, Q_{p_k})$ many solutions. We briefly introduce the framework of the polyhedral homotopy continuation. For a polynomial $p\in \mathbb{C}[z_1,\dots, z_k]$ and its support set $A_{p}$, we know that
		\[p(z) = \sum\limits_{a\in A_p}c_az^a.\]
		Let $\ell_p:A_p\rightarrow \mathbb{R} $ be a function defined on every lattice point in $A_p$,
		we define
		\be\label{eq:lifting_p}
		\overline{p}(z,t)=\sum\limits_{a\in A_p}c_az^a t^{\ell_p(a)}\ee
		which is called a \emph{lifted polynomial} of $p$ by the lifting function $\ell_p$.
		Lifting all polynomials $p_1,\dots, p_k$ in $\mathscr{P}$ gives a lifted system $\overline{\mathscr{P}}(z,t)$. Note that $\overline{\mathscr{P}}(z,1)=\mathscr{P}$. A solution of $\overline{\mathscr{P}}$ can be expressed by a Puiseux series $z(t)=(z_1(t),\dots, z_k(t))$ where 
		\[z_i(t)=t^{\alpha_i}y_i+ \text{(higher order terms)}\]
		for some rational number $\alpha_i$ and a nonzero constant $y_i$. As $z(t)$ is a solution for the lifted system, plugging $z(t)$ into each $p_j$ gives
		\[\overline{p}_j(z(t),t) =\sum\limits_{a\in A_{p_j}}c_a y_i^a t^{\langle a,\alpha\rangle+\ell_{p_j}(a)}+\text{(higher order terms)}.\] 
		Dividing by $t^{\langle a,\alpha\rangle +\ell_{p_j}(a)}$ and letting $t=0$, we have a start system $\mathscr{Q}$. 
		The solutions for $\mathscr{Q}$ can be obtained from the branches of the algebraic function $z(t)$ near $t=0$. 
		The homotopy $\mathscr{H}(z,t)$ joining $\mathscr{P}$ and $\mathscr{Q}$ has %tracks
		$MV(Q_{p_1},\dots, Q_{p_k})$ many paths as $t$ varies from $0$ to $1$. 
		It is motivated from Theorem \ref{thm:BernsteinThm} that a polynomial system supported on $A_{p_1},\dots, A_{p_k}$ has at most $MV(Q_{p_1},\dots, Q_{p_k})$ many isolated solutions in the torus $(\mathbb{C}\setminus \mathbf{0})^k$. 
		Polyhedral homotopy continuation is implemented robustly in many software \texttt{HOM4PS2} \cite{lee2008hom4ps}, \texttt{PHCpack} \cite{verschelde1999algorithm} and \texttt{HomotopyContinuation.jl} \cite{breiding2018homotopycontinuation}.
		
		\begin{remark}\label{rmk:homotopyRemarks}
			\begin{enumerate}
				\item Even in the case that the number of solutions is smaller than the mixed volume, the polyhedral homotopy continuation algorithm may find all complex solutions.
				\item For GNEPs, there may exist complex KKT tuples outside of the torus.
				For instance, when there are KKT points with inactive constraints, Lagrange multipliers according to these constraints are $0$.
				Theoretically, the polyhedral homotopy continuation aims on finding roots in the torus $(\mathbb{C}\setminus \{\boldsymbol{0}\})^k$.
				However, actual implementations are designed to find roots outside the torus by adding a small perturbation on the constant term;
				see \cite{li1996bkk} for details.
				In Example~\ref{ep:uncNEP}, we give an example where the homotopy continuation successfully finds all complex solutions to a system, 
				while the mixed volume is strictly greater than the number of solutions, and there exist roots outside the torus.
			\end{enumerate}
		\end{remark}
		
		In summary, we give the general framework of polyhedral homotopy continuation for solving a polynomial system in the following:
		\begin{alg}
			\label{ag:polyhedralHomotopy} \rm
			For a system of polynomial equations $\mathscr{P}=\{p_1,\dots, p_k\}$,
			do the following:
			\begin{itemize}
				
				\item [Step~1]
				For each $i=1\ddd k$,
				choose a function $\ell_{p_i}:A_{p_i}\rightarrow \mathbb{R}$.
				Then, construct the lifted polynomial $\bar{p}_i(z,t)$ as in (\ref{eq:lifting_p}),
				and define $\overline{\mathscr{P}}(z,t):=\{\bar{p}_1(z,t)\ddd \bar{p}_k(z,t)\}$.
				
				\item [Step~2]
				Construct a start system $\mathscr{Q}$ from the lifted system $\overline{\mathscr{P}}$ by trimming some powers of $t$ and letting $t=0$.
				
				\item [Step~3]
				Starting from $\mathscr{Q}$, track $MV(Q_{p_1},\dots, Q_{p_k})$ many paths from $t=0$ to $t=1$ with Puiseux series solutions $z(t)$ obtained near $t=0$.

			\end{itemize}
		\end{alg}

		As the polyhedral homotopy continuation approximates the roots of a system numerically, 
		{\it a posteriori} certifications are usually applied to verify the output obtained by numerical solvers,
		such as the Smale's $\alpha$-theory \cite[Chapter 8]{blum2012complexity} and interval arithmetic \cite[Chapter 8]{moore2009introduction}. 
		There are multiple known implementations for these methods. For $\alpha$-theory certification, one can use \texttt{alphaCertified} \cite{hauenstein2011alphacertified} or \texttt{NumericalCertification} \cite{https://doi.org/10.48550/arxiv.2208.01784}.
		For certification using interval arithmetic,  software \texttt{NumericalCertification} and \texttt{HomotopyContinuation.jl}  \cite{breiding2020certifying} are available.

		\subsection{Basic concepts in polynomial optimization}
		For the set of real polynomials $\mc{H} = \{h_1, \ldots, h_s\}$ in $z:=(z_1,\dots, z_k)$, the ideal generated by $\mc{H}$ is
		\[
		\idl[\mc{H}]:= h_1\cdot\mathbb{R}[z] + \cdots + h_s \cdot \mathbb{R}[z].
		\]
		For a nonnegative integer $d$,
		the $d$-truncation of $\idl[\mc{H}]$ is
		\[
		\idl[\mc{H}]_{d} \, := \idl[\mc{H}]\cap\re[z]_{d}.
		\]
		A polynomial $\sigma \in \re[z]$ is said to be a \emph{sum of squares (SOS)}
		if $\sigma = \sigma_1^2+\cdots+\sigma_l^2$ for some polynomials $\sigma_i \in\nfR[z]$.
		The set of all SOS polynomials in $z$ is denoted as $\Sigma[z]$.
		For a degree $d$, we denote the truncation
		\[
		\Sigma[z]_d \, := \, \Sigma[z] \cap \nfR[z]_d.
		\]
		For a set {$\mc{Q}=\{q_1,\ldots,q_t\}$} of polynomials in $z$,
		its quadratic module is the set
		\[
		\qmod[\mc{Q}] \, := \,  \Sigma[z] +  q_1 \cdot \Sigma[z] + \cdots + q_t \cdot  \Sigma[z].
		\]
		Similarly, we denote the truncation of $\qmod[\mc{Q}]$
		\[
		\qmod[\mc{Q}]_{2d} \, := \, \Sigma[z]_{2d} + q_1\cdot \Sigma[z]_{2d-\deg(q_1)}
		+\cdots+q_t\cdot\Sigma[z]_{2d-\deg(q_t)}.
		\]
		The tuple $\mc{Q}$ determines the basic closed semi-algebraic set
		\begin{equation*}
			\mathcal{S}(\mc{Q}) \, := \,  \{z \in \nfR^k\mid q_1(z) \ge  0, \ldots, q_t(z) \ge  0  \}.
		\end{equation*}
		Moreover, for {$\mc{H}=\{h_1\ddd h_s\}$}, its real zero set is
		\[\mathbf{V}_\mathbb{R}(\mc{H}) := \mathbf{V}(\mc{H})\cap \mathbb{R}^k= \{z \in\re^k\mid h_1(z)=\cdots=h_s(z)=0\}.\]
		The set $\idl[\mc{H}]+\qmod[\mc{Q}]$ is said to be \emph{archimedean}
		if there exists $\rho \in \idl[\mc{H}]+\qmod[\mc{Q}]$ such that the set $\mathcal{S}(\{\rho\})$ is compact.
		If $\idl[\mc{H}]+\qmod[\mc{Q}]$ is archimedean, then
		$\mathbf{V}_\mathbb{R}(\mc{H})\cap\mathcal{S}(\mc{Q})$ must be compact.
		Conversely, if $\mathbf{V}_\mathbb{R}(\mc{H})\cap\mathcal{S}(\mc{Q})$ is compact, say,
		$\mathbf{V}_\mathbb{R}(\mc{H})\cap\mathcal{S}(\mc{Q})$ is contained in the ball $R -\|z\|^2 \ge 0$,
		then $\idl[\mc{H}]+\qmod[\mc{Q}\cup \{R -\|z\|^2\}]$ is archimedean and $\mathbf{V}_\mathbb{R}(\mc{H})\cap\mathcal{S}(\mc{Q}) = \mathbf{V}_\mathbb{R}(\mc{H})\cap\mathcal{S}(\mc{Q}\cup \{R -\|z\|^2\})$.
		Clearly, if $f \in \idl[\mc{H}]+\qmod[\mc{Q}]$, then
		$f \ge 0$ on $\mathbf{V}_\mathbb{R}(\mc{H}) \cap \mathcal{S}(\mc{Q})$.
		The reverse is not necessarily true.
		However, when $\idl[\mc{H}]+\qmod[\mc{Q}]$ is archimedean,
		if the polynomial $f > 0$ on $\mathbf{V}_\mathbb{R}(\mc{H})\cap\mathcal{S}(\mc{Q})$, then $f \in\idl[\mc{H}]+\qmod[\mc{Q}]$.
		This conclusion is referenced as Putinar's Positivstellensatz \cite{putinar1993positive}.
		Interestingly, if $f \ge 0$ on $\mathbf{V}_\mathbb{R}(\mc{H})\cap\mathcal{S}(\mc{Q})$,
		we also have $f\in \idl[\mc{H}]+\qmod[\mc{Q}]$,
		under some standard optimality conditions \cite{nie2014optimality}.
		
		\emph{Truncated multi-sequences} (tms) are useful for characterizing the duality of nonnegative polynomials.
		For an integer $d\ge 0$, a real vector
		$y=(y_{\alpha})_{\alpha\in\mathbb{N}_{2d}^k}$ is called a
		tms of degree $2d$.
		For a polynomial $p(z) = \sum_{ \af \in \N^k_{2d} } p_\af z^\af$,
		define the operation
		\be \label{<p,y>}
		\langle p, y \rangle \,  \coloneqq  \,
		{\sum}_{ \af \in \N^k_{2d} } p_\af y_\af.
		\ee
		The operation $\langle p, y \rangle$ is bilinear in $p$ and $y$.
		Moreover, for a polynomial $q \in \re[x]_{2s}$ ($s \le d$),
		and a degree $t\le d - \lceil \deg(q)/2 \rceil$,
		the $d$th order {\it localizing matrix} of $q$ for $y$
		is the symmetric matrix $L_{q}^{(d)}[y]$ such that
		(the $vec(a)$ denotes the coefficient vector of $a$)
		\be \label{df:Lf[y]}
		\langle qa^2, y \rangle  \, =  \,
		vec(a)^T \big( L_{q}^{(d)}[y]  \big) vec(a)
		\ee
		for all $a \in \re[x]_t$.
		When $q=1$ (the constant one polynomial),
		the localizing matrix $L_{q}^{(d)}[y]$
		becomes the {$d$th order} \textit{moment matrix}
		$M_t[y]  \coloneqq   L_{1}^{(d)}[y]$.
		We refer to \cite{LauICM,laserre2015introduction,HenLas05} for more details and applications about tms and localizing matrices.
		In Section~\ref{sc:momsos}, SOS polynomials and localizing matrices are exploited for solving polynomial optimization problems.
		
		\section{An algorithm for finding GNEs}
		\label{sc:alg}
		In this section, we study an algorithm for finding GNEs based on the polyhedral homotopy continuation and the Moment-SOS semidefinite relaxations. First, we propose a framework for solving GNEPs.
		Then, we discuss the polyhedral homotopy continuation for solving the complex KKT systems for GNEPs of polynomials,
		and the Moment-SOS relaxations for selecting GNEs from the set of KKT points.
		
		For the GNEP of polynomials,
		we consider its complex KKT system (\ref{eq:KKTep}).
		Let $m:=m_1+\cdots+m_N$ and define $\mc{K}_{\cpx}\subseteq\cpx^n\times\cpx^m$ as a finite set of complex KKT tuples,
		i.e.,  every point in $\mc{K}_{\cpx}$ solves the system (\ref{eq:KKTep}).
		We further define
		\[\mc{K}:=\left\{(x,\lambda)\in\mc{K}_{\cpx}\cap (\re^n\times\re^m)\left|
		\begin{array}{c}
			\lambda_{i,j}\ge0, g_{i,j}(x)\ge0\\
			\forall \, i\in\{1,\dots, N\},\, j\in\mc{I}_i 
		\end{array}\right.\right\},\]
		\be\label{eq:mcP} 
		\mc{P}:=\{x\in\re^n\mid \text{there is } \lambda\in\re^m\ \text{ such that } (x,\lmd)\in\mc{K}\}.\ee
		Then, $\mc{K}$ and $\mc{P}$ are sets of KKT tuples and KKT points respectively.
		When the GNEP is convex,
		all points in $\mc{P}$ are GNEs.
		Furthermore, when $\mc{K}_{\cpx}$ is the set of all complex KKT tuples and some constraint qualifications hold at every GNE,
		the $\mc{P}$ is the set of all GNEs if it is nonempty,
		or the nonexistence of GNEs can be certified by the emptiness of $\mc{P}$.
		However, when there is no convexity assumed for the GNEP,
		the KKT conditions are usually not sufficient for $x\in\mc{P}$ being a GNE.
		
		Suppose the GNEP is not convex.
		Let $u=(u_i,\umi)\in\mc{P}$.
		For each $i\in \{1,\dots, N\}$, we consider the following optimization problem:
		\be
		\label{eq:checking}
		\left\{ \begin{array}{ccl}
			\delta_i:= & \min\limits_{\xpi\in \Rni}  &  \fpi(x_i,u_{-i})-\fpi(u_i,u_{-i}) \\
			& \st & g_{i,j}(x_i,u_{-i})  = 0 \ (j\in\mc{E}_i), \\
			&     & g_{i,j}(x_i,u_{-i})  \ge 0 \ (j\in\mc{I}_i).
		\end{array} \right.
		\ee
		Then, $u$ is a GNE if and only if every $\delta_i\ge0$,
		i.e.,  $u_i$ minimizes (\ref{eq:checking}) for each $i$.
		If $\delta_i<0$ for some $i$,
		then $u$ is not a GNE.
		For such a case, suppose (\ref{eq:checking}) has a minimizer $v_i$.
		Then it is clear that \be\label{eq:vi-} 
		f_i(v_i,\xmi)-f_i(\xpi,\xmi)\ge0\ee
		holds with $x=x^\star$ at any GNE $x^\star\in\mc{P}$ such that $v_i\in X_i(\xmi^\star)$ (e.g., the $v_i\in X_i(\xmi^\star)$ holds at all NEs for NEPs).
		However, (\ref{eq:vi-}) does not hold at $x=u$.
		Therefore, we propose the following algorithm for finding GNEs.
		\begin{alg}
			\label{ag:selectGNE} \rm
			For the GNEP of polynomials, do the following:
			
			\begin{itemize}
				
				\item [Step~0] Let $S:=\emptyset$ and $V_i:=\emptyset$ for all $i\in\{1,\dots, N\}$. 
				
				\item [Step~1] Solve the complex KKT system (\ref{eq:KKTep}) for a set of complex solutions $\mc{K}_{\cpx}$.
				Let $\mc{P}$ be the set given as in (\ref{eq:mcP}).
				\item [Step~2]
				If $\mc{P}\ne \emptyset$,
				then select $u\in\mc{P}$, let $\mc{P}:=\mc{P}\setminus\{u\}$, and proceed to the next step.
				Otherwise, output the set $S$ (possibly empty) of GNEs and stop.
				
				\item [Step~3]
				If $V_i=\emptyset$ for all $i\in\{1,\dots, N\}$, or $f_i(v_i,\umi)-f_i(u_i,\umi)\ge0$ for all $i\in\{1,\dots, N\}$ and for all $v_i\in V_i\cap X_i(\umi)$,
				then go to the next step.
				Otherwise, go back to Step~2.
				
				\item [Step~4]
				For each $i\in\{1,\dots, N\}$, solve the polynomial optimization problem (\ref{eq:checking}) for a minimizer $v_i$.
				If there exists $i\in\{1,\dots, N\}$ such that $\dt_i<0$, let $V_i:=V_i\cup\{v_i\}$ for all such $i$.
				Otherwise, $u$ is a GNE and let $S:=S\cup\mc\{u\}$.
				Then, go back to Step~2.
			\end{itemize}
		\end{alg}
		
		In Section~\ref{sc:homotopy},
		we show how to find the set of complex solutions for the system (\ref{eq:KKTep}) using the polyhedral homotopy continuation.
		Since the polyhedral homotopy tracks mixed volume many paths,
		$\mc{P}$ must be a finite set (possibly empty).
		Therefore, Algorithm~\ref{ag:selectGNE} must terminate within finitely many loops.
		Moreover, if $\mc{K}_{\cpx}$ is the set of all complex KKT tuples,
		i.e.,  Algorithm~\ref{ag:polyhedralHomotopy} finds all complex solutions for (\ref{eq:KKTep}),
		then Algorithm~\ref{ag:selectGNE} will either find all GNEs,
		or detect the nonexistence of GNEs.
		This is the case when Algorithm~\ref{ag:polyhedralHomotopy} finds the mixed volume many complex solutions for (\ref{eq:KKTep}).
		In Section~\ref{sc:mv},
		we show that when the GNEP is defined by generic dense polynomials,
		Algorithm~\ref{ag:polyhedralHomotopy} can find mixed volume many solutions for (\ref{eq:KKTep}).
		The following result is straightforward:
		\begin{theorem}
			\label{tm:mvfound}
			For the GNEP,
			suppose $|\mc{K}_{\cpx}|$ equals the mixed volume of (\ref{eq:KKTep}).
			If $S$ produced by Algorithm~\ref{ag:selectGNE} is nonempty,
			then $S$ is the set of all GNEs.
			Otherwise, the GNEP does not have any GNE.  
		\end{theorem}
		
		\subsection{The polyhedral homotopy method for finding KKT tuples}
		\label{sc:homotopy}
		In this subsection, we explain how the polyhedral homotopy continuation is applied to find complex solutions for the system (\ref{eq:KKTep}).
		For each $i\in \{1,\dots, N\}$, denote the set of polynomials in variables of $x$ and $\lmd_i$ as
		\begin{multline*}
			F_i(x,\lmd_i):=
			\left\{\nabla_{x_i} f_i(x)-\sum\nolimits_{j=1}^{m_i}\lambda_{i,j}\nabla_{x_i} g_{i,j}(x)\right\}\\
			\cup \left\{\lambda_{i,j}g_{i,j}(x)\mid j\in\mc{I}_i\right\}\cup\left\{g_{i,j}(x)\mid j\in\mathcal{E}_i\right\}.
		\end{multline*}
		We define a system
		\be\label{eq:Fx} 
		F(x,\lmd):=\bigcup_{i=1}^N F_i(x,\lmd_i).\ee
		Then, $F(x,\lambda)$ is a system with $n+m$ polynomial equations in $n+m$ variables,
		and we use Algorithm~\ref{ag:polyhedralHomotopy} to solve $F(x,\lambda)=0$ by letting $z:=(x,\lmd)$ and $\mathscr{P}(z):=F(z)$.

		The example below shows detail of applying the homotopy method for finding KKT tuples from an actual NEP problem.
		\begin{example}\label{ep:uncNEP}
			Consider the two-player unconstrained NEP
			\[
			\begin{array}{cllcl}
				\min\limits_{x_{1} \in \re^1}& \frac{1}{2}x_1^2x_2^3-x_1x_2^2-2x_1x_2 &\vline&
				\min\limits_{x_{2} \in \re^1}& \frac{1}{2}x_1^3x_2^2-x_1^2x_2-2x_1x_2.
			\end{array}
			\]
			For this problem,
			the complex KKT system reduces to vanishing the gradients $\nabla_{x_1}f_1$ and $\nabla_{x_2}f_2$,
			i.e., 
			we have \[F=\{\nabla_{x_1}f_1,\nabla_{x_2}f_2\}=\{x_1x_2^3-x_2^2-2x_2,x_1^3x_2-x_1^2-2x_1\}.\]
			Considering a lifted system with generic lifting functions $\ell_{f_1}$ and $\ell_{f_2}$, we have
			\[\begin{array}{l}
				\overline{F}(x,t)=\{x_1x_2^3t^{\ell_{f_1}(1,3)}-x_2^2t^{\ell_{f_1}(0,2)}-2x_2t^{\ell_{f_1}(0,1)},\\
				\qquad \qquad \qquad \qquad x_1^3x_2t^{\ell_{f_2}(3,1)}-x_1^2t^{\ell_{f_2}(2,0)}-2x_1t^{\ell_{f_2}(1,0)}\}
			\end{array}\]
			such that $\overline{F}(x,1)=F$. Also, we get a start system $\overline{F}(x,0)$ after trimming some powers of $t$.
			The system $F$ has the mixed volume equal to $8$.
			Therefore, the polyhedral homotopy continuation provided $8$ paths to track and found $6$ numerical solutions to $F$:
			\[(x_1,x_2):=\left\{\begin{array}{l}
				(-0.76069+0.857874i ,-0.76069+0.857874i),\\
				(-0.76069-0.857874i ,-0.76069-0.857874i),\\
				(1.52138,1.52138),\ (0,-2),\ (-2,0),\ (0,0).
			\end{array} \right.
			\]
			% of $F$ are found.
			Indeed, using the software {\tt Macaulay2} \cite{macaulay2},
			we verified that the system $F(x)=0$ has exactly $6$ complex solutions.
		\end{example}
		
		\begin{remark}
			\begin{enumerate}
				\item Note that the system $F(x)=0$ has $6$ complex solutions, which is strictly less than its mixed volume.
				This shows that the polyhedral homotopy continuation may find all complex solutions to the polynomial system even if the number of solutions is smaller than the mixed volume. 
				\item As presented in this example,
				the polyhedral homotopy continuation may find solutions outside the torus in the actual implementation.
			\end{enumerate}
		\end{remark}
		
		In Section~\ref{sc:mv}, we show that under the genericity assumption, the polyhedral homotopy continuation provides the optimal number of paths for finding complex KKT points. 
		In this case, the polyhedral homotopy continuation guarantees finding all complex KKT points,
		hence the complete collection of GNEs can be obtained by Algorithm~\ref{ag:selectGNE}.
		
		\subsection{The Moment-SOS relaxation for selecting GNEs}\label{sc:momsos}
		In this sequel,
		we discuss how to solve the polynomial optimization (\ref{eq:checking}).
		For each $i$, denote 
		\[\theta_i(x_i):=\fpi(x_i,u_{-i})-\fpi(u_i,u_{-i}),\]
		\[\Phi_i(x_i):=\{g_{i,j}(x_i,\umi)\mid j\in\mc{E}_i\},\]
		\[\Psi_i(x_i):=\{g_{i,j}(x_i,\umi)\mid j\in\mc{I}_i\}.\]
		Denote the degree
		\be\label{eq:d0}
		d_i:=\max\{\lceil\deg(\theta_i)/2\rceil, \lceil\deg(\Phi_i(x_i))/2\rceil,\lceil\deg(\Psi_i(x_i))/2\rceil\},
		\ee
		{where $\deg(\Phi_i(x_i)):=\max\{\deg(g_{i,j}(x_i,u_{-i}))\mid j\in \mc{E}_i\}$,
			and $\deg(\Psi_i(x_i))$ is similarly defined.}
		For $d\ge d_i$, recall that the tms $y\in\mathbb{R}^{\mathbb{N}^{n_i}_{2d}}$ and localizing matrices $L_q^{(d)}[y]$ are given by (\ref{<p,y>}) and (\ref{df:Lf[y]}) respectively.
		The $d$th moment relaxation for \reff{eq:checking} is
		\be
		\label{eq:d-mom}
		\left\{
		\baray{rl}
		\vartheta_i^{(d)}  \, :=  \,\min\limits_{ y\in\mathbb{R}^{\mathbb{N}^{n_i}_{2d}} } & \lip \theta_i ,y \rip\\
		\st  & L_{p}^{(d)}[y] \succeq 0 \, (p\in\Phi_i),\ L_{q}^{(d)}[y] = 0 \, (q\in\Psi_i),\\
		& y_0=1, \, M_d[y] \succeq 0,  \\
		\earay
		\right.
		\ee
		Its dual optimization problem is the $d$th SOS relaxation
		\be
		\label{eq:d-sos}
		\left\{
		\baray{ll}
		\max & \gamma\\
		\st & \theta_i -\gamma \in \idl(\Psi_i)_{2d}+ \qmod(\Phi_i)_{2d}. \\
		\earay
		\right.
		\ee
		Both \reff{eq:d-mom}-\reff{eq:d-sos} are semidefinite programs that can be efficiently solved by some well-developed methods and software (e.g., {\tt SeDuMi} \cite{SeDuMi}).
		By solving the relaxations \reff{eq:d-mom}-\reff{eq:d-sos}
		for $d=d_0, d_0+1, \ldots$, we get the following algorithm, named Moment-SOS hierarchy \cite{Las01},
		for solving (\ref{eq:checking}).
		
		\begin{alg} \label{ag:checkingsdp} \rm
			For the given $u \in \mc{P}$
			and the $i$th player's optimization \reff{eq:checking}. Initialize $d := d_i$.
			\begin{itemize}
				
				\item [Step~1]
				Solve the moment relaxation (\ref{eq:d-mom})
				for the minimum value $\vartheta_i^{(d)}$ and a minimizer $y^\star$.
				If $\vartheta_i^{(d)}\ge0$,
				then $\delta_i=0$ and stop;
				otherwise, go to the next step.
				
				\item [Step~2]
				Let $t:=d_i$ as in \reff{eq:d0}.
				If $y^\star$ satisfies the rank condition
				\be \label{eq:flatrank}
				\Rank{M_t[y^\star]} \,=\, \Rank{M_{t-d_i}[y^\star]} ,
				\ee
				then extract a set $U_i$ of $r :=\Rank{M_t(y^\star)}$ minimizers for \reff{eq:checking} and stop.
				
				\item [Step~3]
				If \reff{eq:flatrank} fails to hold and $t < d$,
				let $t := t+1$ and then go to Step~2;
				otherwise, let $d := d+1$ and go to Step~1.
				
			\end{itemize}
		\end{alg}

		The rank condition \reff{eq:flatrank} is called {\it flat truncation} \cite{nie2013certifying}.
		It is a sufficient (and almost necessary) condition for checking finite convergence of the Moment-SOS hierarchy.
		Indeed, the Moment-SOS hierarchy has finite convergence
		if and only if the flat truncation is satisfied for some relaxation orders,
		under some generic conditions~\cite{nie2013certifying}.
		When \reff{eq:flatrank} holds, the method in \cite{HenLas05}
		can be used to extract $r$ minimizers for (\ref{eq:checking}).
		The method is implemented in the software {\tt GloptiPoly3}~\cite{GloPol3}.
		We refer to \cite{HenLas05,nie2013certifying} and
		\cite[Chapter 6]{laserre2015introduction} for more details.

		The convergence properties of Algorithm~\ref{ag:checkingsdp} are as follows.
		By solving the hierarchy of relaxations \reff{eq:d-mom} and \reff{eq:d-sos},
		we get a monotonically increasing
		sequence of lower bounds $\{\vartheta_d\}_{d=d_0}^{\infty}$
		for the minimum value $\vartheta_{\min}$, i.e., 
		\[
		\vartheta_{d_0} \le \vartheta_{d_0+1} \le \cdots \le \vartheta_{\min}.
		\]
		When $\idl(\Psi_i)_{2d}+ \qmod(\Phi_i)_{2d}$ is archimedean, we have
		$\vartheta_d \to \vartheta_{\min}$ as $d \to \infty$ \cite{Las01}.
		If $\vartheta_d = \vartheta_{\min}$ for some $d$,
		the relaxation (\ref{eq:d-mom}) is said to be exact
		for solving (\ref{eq:checking}). For such a case,
		the Moment-SOS hierarchy is said to have finite convergence.
		This is guaranteed when the archimedean and some optimality conditions hold (see~\cite{nie2014optimality}).
		Although there exist special polynomials such that
		the Moment-SOS hierarchy fails to have finite convergence,
		such special problems belong to a set of measure zero in the space of
		input polynomials \cite{nie2014optimality}.
		We refer to \cite{LasICM,laserre2015introduction,LauICM,nie2014optimality} for more work on polynomial and moment optimization.

		\section{The mixed volume of GNEPs}
		\label{sc:mv}
		For a polynomial system,
		if the set of its complex solutions is zero-dimensional,
		then the algebraic degree of the polynomial system
		counts the number of complex solutions for the system.
		In this section, we prove that under some genericity assumptions on the GNEP,
		the mixed volume of the complex KKT system (\ref{eq:KKTep}) equals its algebraic degree.
		Throughout this section, we have a GNEP of polynomials consisting of dense polynomials of certain degrees. 
		Without loss of generality, we assume $\mc{I}_i=\emptyset$ for all $i\in\{1,\dots, N\}$,
		i.e.,  all players only have equality constraints,
		for the convenience of our discussion.
		Note that if there exist inequality constraints,
		then all following results still hold by enumerating the active constraints.
		For a tuple $d:=(d_1\ddd d_N)$ of nonnegative integers,
		the $\cpx[x]_d$ represents the space of polynomials whose degree in $x_i$ is not greater than $d_i$.
		
		Recall that we say a system of polynomials is Bernstein generic if the number of isolated solutions in the complex torus equals its mixed volume.
		The main result of this section is the following:
		\begin{theorem}\label{thm:mixedVolumeOfGNEP}
			Consider the GNEP of polynomials given as in (\ref{eq:GNEP}).
			For each $i$, let $d_{i,0}\ddd d_{i,m_i}\in\N^N$ be tuples of nonnegative integers.
			Suppose all $f_i$ and $g_{i,j}$ are generic dense polynomials in $\cpx[x]_{d_{i,0}}$ and $\cpx[x]_{d_{i,j}}$ respectively.
			Then, the complex KKT system (\ref{eq:Fx}) is Bernstein generic.
		\end{theorem}
		We first introduce some basic notation and useful lemmas before showing Theorem \ref{thm:mixedVolumeOfGNEP}.
		Let $w_1,\dots, w_N$ be weight vectors such that
		$$w_i=(\mathbf{0},\dots,\mathbf{0},(w_{i,1},\dots, w_{i,n_i},v_{i,1},\dots,v_{i,m_i}),\mathbf{0},\dots,\mathbf{0}),$$
		where $w_{i,k}$ and $v_{i,j}$ are weights for variables $x_{i,k}$ and $\lambda_{i,j}$ respectively. 
		Define $w:=\sum_{i=1}^N w_i$.
		Each $w_i$ is the weight vector for the system $F$ applied only for $x_i$ and $\lambda_{i}$ variables.

		The idea for the proof of the result is inspired by the paper \cite{breiding2020euclidean}. 
		We introduce the lemmas established from the paper that will be used for the proof of the Theorem \ref{thm:mixedVolumeOfGNEP}.

		\begin{lemma}\cite[Lemma 8]{breiding2020euclidean}
			Let $p$ be a polynomial in $\mathbb{C}[x]$ and $w\in \mathbb{Z}^n$ be an integer vector. If $\frac{\partial p^w}{\partial x_i}\ne 0$, then $\frac{\partial p^w}{\partial x_i}=\left(\frac{\partial p}{\partial x_i}\right)^w$ and $h^\star(A_{\frac{\partial p}{\partial x_i}})=h^\star(A_p)-w_i$.
		\end{lemma}
		
		\begin{lemma}\label{lm:geneuler}
			Let $p$ be a polynomial in $\mathbb{C}[x]_d$ and $w=\sum_{i=1}^N w_i$ be a weight vector in $\mathbb{Z}^n$.
			Then, we have
			\[h_{w_i}^\star(A_p)\cdot p^w=\sum\limits_{k=1}^{n_i}w_{i,k}x_{i,k}\frac{\partial p^w}{\partial x_{i,k}}.\]
		\end{lemma}
		\begin{proof}
			For a monomial $x^a$, note that 
			$x_{i,k}\frac{\partial x^a}{\partial x_{i,k}}=a_{i,k}x^a$. Therefore, we have
			\[\sum_{k=1}^{n_i}w_{i,k}x_{i,k}\frac{\partial x^a}{\partial x_{i,k}}=\sum_{k=1}^{n_i}w_{i,k}a_{i,k}x^a=\langle w_i,a\rangle x^a.
			\]
			Summing over all monomials in $p^w$, we get
			\[h_{w_i}^\star(A_{p^w})\cdot p^w=\sum\limits_{k=1}^{n_i}w_{i,k}x_{i,k}\frac{\partial p^w}{\partial x_{i,k}}.\]
			Noting the fact that $h_{w_i}^\star(A_p)=h_{w_i}^\star(A_{p^w})$, we get the desired result.
		\end{proof}
		Note that Lemma~\ref{lm:geneuler} is a generalization of Euler's formula for quasihomogeneous polynomials mentioned in \cite[Lemma 9]{breiding2020euclidean}. 
		We are now ready to prove Theorem \ref{thm:mixedVolumeOfGNEP}.
		
		\begin{proof}[Proof of Theorem~\ref{thm:mixedVolumeOfGNEP}]
			Recall the second part of Theorem \ref{thm:BernsteinThm} that a polynomial system is Bernstein generic if and only if the facial system has no root in the complex torus for any nonzero vector $w$.
			For each polynomial $p$ and a weight vector $w$, let $\mathbf{I}^w_i(p)$ and $\mathbf{J}^w_i(p)$ be partition sets of indices $\{1,\dots,n_i\}$ for each $i$,
			such that $\frac{\partial p^w}{\partial x_{i,j}}\ne0$ if $j\in \mathbf{I}^w_i(p)$ and $\frac{\partial p^w}{\partial x_{i,j}}=0$ if $j\in \mathbf{J}^w_i(p)$. 
			That is, $\mathbf{I}^w_i(p)$ is the set of all labels $j$ such that the variable $x_{i,j}$ appears in $p^w$, and $\mathbf{J}^w_i(p)=\{1 \ddd n_i\}\setminus \mathbf{I}^w_i(p)$.
			Also, we let 
			\[\mathbf{I}_i^w=\mathbf{I}_i^w(f_i)\cup\left(\bigcup^{m_i}_{j=1}\mathbf{I}_i^w(g_{i,j})\right),\quad \mathbf{I}^w=\bigcup_{i=1}^N \mathbf{I}_i^w \mbox{,\quad and\quad}\hat{n}_i=\left|\mathbf{I}_i^w\right|.\]
			Furthermore, let $\hat{n}:=\sum_{i=1}^N \hat{n}_i$.
			It is clear that if $j\in \mathbf{J}^w_i(p)$ and $a\in A^w_p$, then $a_{i,j}=0$. 
			Hence, 
			we may consider $p^w$ as a polynomial in $\mathbb{C}[\mathbf{I}^w]:=\mathbb{C}[x_{i,j}\mid j \in \mathbf{I}^w_i(p),i = 1,\dots, N]$. 
			Note that if $p$ is a generic polynomial, then $p^w$ can also be considered as a generic polynomial for a given support.

			In the following, for a fixed weight vector $w$, we show that there are no roots for the facial system $F^w$ in the torus $(\mathbb{C}\setminus\{\boldsymbol{0}\})^n$. 
			Consider the facial system of $F_i$, say,
			$$F_i^w=\left\{(\nabla_{x_i}f_i-\sum_{j=1}^{m_i}\lambda_{i,j}\nabla_{x_i}g_{i,j})^w,\, g_{i,1}^w,\, \dots,\, g_{i,m_i}^w\right\},$$ 
			where \[\nabla_{x_i}f_i-\sum_{j=1}^{m_i}\lambda_{i,j}\nabla_{x_i}g_{i,j}=\left\{\frac{\partial f_i}{\partial x_{i,1}}-\sum_{j=1}^{m_i}\lambda_{i,j}\frac{\partial g_{i,j}}{\partial x_{i,1}},\dots,\frac{\partial f_i}{\partial x_{i,n_i}}-\sum_{j=1}^{m_i}\lambda_{i,j}\frac{\partial g_{i,j}}{\partial x_{i,n_i}}\right\}.\]
			For each $k\in\{1,\dots,n_i\}$,
			let $A_{\partial_{i,k}}$ be the support of $\frac{\partial f_i}{\partial x_{i,k}}-\sum_{j=1}^{m_i}\lambda_{i,j}\frac{\partial g_{i,j}}{\partial x_{i,k}}$.
			Then, we have 
			\[h^\star(A_{\partial_{i,k}})=\min\left\{h^\star\left(A_{\frac{\partial f_i}{\partial x_{i,k}}}\right),\min\limits_{j=1,\dots,m_i}\left\{h^\star\left(A_{\frac{\partial g_{i,j}}{\partial x_{i,k}}}\right)+v_{i,j}\right\}\right\}.\]
			In the above, we recall that for a polynomial $p$, the $h^\star(A_{p})$ is the minimum of $h(a):=\lip w,a \rip$ over $a\in A_p$; see Section~\ref{subsection:mixedVolume} for more details.
			Depending on where the value of $h^\star(A_{\partial_{i,k}})$ is attained, there are the following three cases:
			\begin{enumerate}[label=\textbf{(\alph*)}]
				\item  \label{casea}
				$h^\star(A_{\partial_{i,k}})=\min\limits_{j=1,\dots,m_i}\left\{h^\star\left(A_{\frac{\partial g_{i,j}}{\partial x_{i,k}}}\right)+v_{i,j}\right\}<h^\star\left(A_{\frac{\partial f_i}{\partial x_{i,k}}}\right)$,
				\item \label{caseb}
				$h^\star(A_{\partial_{i,k}})=h^\star\left(A_{\frac{\partial f_i}{\partial x_{i,k}}}\right)<\min\limits_{j=1,\dots,m_i}\left\{h^\star\left(A_{\frac{\partial g_{i,j}}{\partial x_{i,k}}}\right)+v_{i,j}\right\}$,
				\item \label{casec} $h^\star(A_{\partial_{i,k}})=h^\star\left(A_{\frac{\partial f_i}{\partial x_{i,k}}}\right)=\min\limits_{j=1,\dots,m_i}\left\{h^\star\left(A_{\frac{\partial g_{i,j}}{\partial x_{i,k}}}\right)+v_{i,j}\right\}$.
			\end{enumerate}
			Let $M_i^w\subseteq\{1,\dots, m_i\}$ be the set of indices $l$ such that 
			\[h^\star\left(A_{\frac{\partial g_{i,l}}{\partial x_{i,k}}}\right)+v_{i,l} = \min\limits_{j=1,\dots,m_i}\left\{h^\star\left(A_{\frac{\partial g_{i,j}}{\partial x_{i,k}}}\right)+v_{i,j}\right\}.\]
			Then, for each $k=1,\dots,n_i$, we have
			\[\left(\frac{\partial f_i}{\partial x_{i,k}}-\sum\limits_{i=1}^{m_i}\lambda_{i,j}\frac{\partial g_{i,j}}{\partial x_{i,k}}\right)^w=\left\{\begin{array}{ll}
				-\sum\limits_{j\in M_i^w}\lambda_{i,j}\frac{\partial g^w_{i,j}}{\partial x_{i,k}}, & \textbf{Case  }\ref{casea}\\
				&\\
				\frac{\partial f_i^w}{\partial x_{i,k}}, & \textbf{Case }\ref{caseb}\\
				&\\
				\frac{\partial {f_i}^w}{\partial x_{i,k}}-\sum\limits_{j\in M_i^w}\lambda_{i,j}\frac{\partial g^w_{i,j}}{\partial x_{i,k}}, & \textbf{Case }\ref{casec}.
			\end{array}\right.\] 
			Note that for a fixed $i$,
			if we consider a generic dense polynomial $p\in \mathbb{C}[x]$ with a fixed multidegree,
			then we have the same support $A_{\frac{\partial p}{\partial x_{i,k}}}$ for $\frac{\partial p}{\partial x_{i,k}}$ for any $k =1,\dots, n_i$.
			Therefore, the values of $h^\star\left(A_{\frac{\partial f_{i}}{\partial x_{i,k}}}\right)$ and $\min\limits_{j=1,\dots,m_i}\left\{h^\star\left(A_{\frac{\partial g_{i,j}}{\partial x_{i,k}}}\right)+v_{i,j}\right\}$ do not depend on the choice of $k=1,\dots, n_i$.
			It means that without loss of generality,
			if we have $h^\star\left(A_{\frac{\partial f_i}{\partial x_{i,k}}}\right)>\min\limits_{j=1,\dots,m_i}\left\{h^\star\left(A_{\frac{\partial g_{i,j}}{\partial x_{i,k}}}\right)+v_{i,j}\right\}$ for some $k\in \{1,\dots,n_i\}$, then so do all other indices in $\{1,\dots, n_i\}$. 
			Furthermore, since we only concern zeros of $F^w$ in the torus,
			we assume that not all polynomials in $F^w$ are divisible by any single variable $x_{i,j}$;
			otherwise, we may divide all polynomials in $F^w$ by a proper power of $x_{i,j}$ until one of them cannot be divided by $x_{i,j}$ any further.

			For each index $i$, 
			let 
			\[U_i^w :=\mathbf{V}(g_{i,1}^w,\dots,g_{i,m_i}^w)\subset \mathbb{C}^{\hat{n}},\]
			\[\text{Jac}_i^w:=\begin{bmatrix}
				\nabla_{x_i}f_i^w(x) & \nabla_{x_i}g_{i,1}^w (x) & \cdots & \nabla_{x_i}g_{i,m_i}^w (x)
			\end{bmatrix},\ \mbox{and}\]
			$$W_i^w:=\left\{x\in \mathbb{C}^{\hat{n}} \mid \text{rank}  (\text{Jac}_i^w)
			\leq m_i\right\}.$$
			Note that the variety $U^w_i$ is defined by the $i$th player's feasibility constraints,
			$\text{Jac}^w_i$ is the transpose of the Jacobian matrix in the variable $x_i$ for the vector $[f^w_i,\ g^w_{i,1},\ \dots,\ g^w_{i,m_i}]^{\top}$,
			and $W^w_i$ is the determinantal variety given by the facial system of complex Fritz-John conditions (see \cite[Section 3]{Nie2022degree}). 
			Recall the assumption that not all polynomials involved are divisible by any single variable.
			For all $2\le l\le m_i$, the hypersurface given by $g^w_{i,l}(x)=0$ intersects the common zeros of $g^w_{i,1}\ddd g^w_{i,l-1}$ without any fixed point when varying the coefficients of $g^w_{i,l}$.
			Thus, by Bertini's theorem (see \cite[Proposition~2.2]{Nie2022degree} for example), the 
			variety $U_i^w$ is of codimension $m_i$ (or possibly empty).
			Then, by a similar argument,
			$\dim (U_i^w\cap W_i^w)\leq\hat{n}-\hat{n}_i$. Also, following the argument in the proof of \cite[Theorem 3.1]{Nie2022degree}, the dimension for $V_{-i}^w:=\bigcap_{k\ne i}(U_k^w\cap W_k^w)$ is not greater than $\hat{n}_i$.

			If $x^\star\in W_i^w$, 
			then there exist $\lambda_{i,0},\dots,\lambda_{i,m_i}\in \mathbb{C}$ such that
			\[\lambda_{i,0}\nabla_{x_i}f_i^w (x^\star) +\lambda_{i,1}\nabla_{x_i}g_{i,1}^w (x^\star) +\cdots+\lambda_{i,m_i}\nabla_{x_i}g_{i,m_i}^w (x^\star) =0.\]
			It means that if $\left(\frac{\partial f_i}{\partial x_{i,k}}-\sum_{j=1}^{m_i}\lambda_{i,j} \frac{\partial g_{i,j}}{\partial x_{i,k}}\right)^w (x^\star) =0$, then $x^\star \in W_i^w$.
			Indeed, all nonzero solutions to the facial system, if they exist, must be in $U_i^w\cap W_i^w$ for all $i=1\ddd N$.
			If $m_i> \hat{n}_i$ for some $i$, then ${ U_i^w\cap}{V}^w_{-i}=\emptyset$ when $g_{i,j}$ are generic polynomials, so $F^w$ does not have any solution. Hence, we may assume that $m_i\leq \hat{n}_i$.
			From now on, we prove the desired statement by considering each of the three cases mentioned above respectively.
			
			\textbf{Case} \ref{casea}. 
			Suppose that there exists $i=1\ddd N$ such that
			\[\left(\frac{\partial f_i}{\partial x_{i,k}}-\sum\limits_{i=1}^{m_i}\lambda_{i,j}\frac{\partial g_{i,j}}{\partial x_{i,k}}\right)^w=-\sum\limits_{j\in M_i^w}\lambda_{i,j}\frac{\partial g^w_{i,j}}{\partial x_{i,k}}.\]
			Without loss of generality, we assume $i=1$.
			Note that if there is a root $(x^\star,\lambda^\star)$ over $(\mathbb{C}\setminus \{\boldsymbol{0}\})$ of $F^w$, then
			\be\label{eq:case1}\sum_{j\in M_1^w}\lambda_{1,j}^\star \frac{\partial g_{1,j}^w}{\partial x_{1,k}}(x^\star)=0.\ee
			Denote by $(\text{Jac}_i^w)^{\circ}$ the submatrix of the rightmost $m_i$ columns of $\text{Jac}_i^w$, and
			$$(W_i^w)^{\circ}:=\{x\in \mathbb{C}^{\hat{n}}\mid \text{rank}  (\text{Jac}_i^w)^{\circ} \leq m_i-1\}.$$ 
			Then the equation (\ref{eq:case1}) implies that $x^\star\in(W_1^w)^{\circ}\cap U_1^w$.
			For a generic $z_{-1}\in\cpx^{\hat{n}-\hat{n}_1}$, \cite[Proposition~2.2]{Nie2022degree} implies that
			the variety $\{x_1\in\cpx^{\hat{n}_1}\mid g_{1,1}^w(x_1,z_{-1})=\dots=g_{1,m_1}^w(x_1,z_{-1})\}$ is smooth, i.e., the matrix $(\text{Jac}_i^w)^{\circ}$ has full column rank at $(x_1,z_{-1})$ for all $x_1\in\cpx^{\hat{n}_1}$.
			So we know $\dim((W_1^w)^{\circ}\cap U_1^w)<\hat{n}-\hat{n}_1$ by 
			\cite[Theorem~1.25]{shafarevichBook}.
			Thus, we have $(W_1^w)^{\circ}\cap U_1^w\cap V^w_{-1}=\emptyset$, which contradicts to the fact that $(x^\star,\lambda^\star)$ is a root of $F^w$.

			\textbf{Case} \ref{caseb}. 
			Suppose that there exists $i=1\ddd N$ such that
			\be\label{eq:caseb}
			\left(\frac{\partial f_i}{\partial x_{i,k}}-\sum\limits_{i=1}^{m_i}\lambda_{i,j}\frac{\partial g_{i,j}}{\partial x_{i,k}}\right)^w=\frac{\partial f^w_{i}}{\partial x_{i,k}}.\ee Without loss of generality, assume that $i=1$.
			We further assume that $m_1\ne 0$ because if $m_1=0$, then it can be considered as a special case of Case \ref{casec}. For a root $(x^\star,\lambda^\star)$ for the facial system,
			we have $\frac{\partial f_1^w}{\partial x_{1,k}}(x^\star)=0$ for all $k=1,\dots,n_1$. 
			Then, $\mathbf{V}(\frac{\partial f_1^w}{\partial x_{1,k}}\mid k=1,\dots, n_1)\cap V^w_{-1}$ has the dimension at most zero due to the genericity of $f_1^w$. 
			Hence, the genericity of $g_{1,j}^w$ concludes that there is no point in $\bigcap_{i=1}^N(W_i^w\cap U_i^w)$ satisfying (\ref{eq:caseb}).
			
			\textbf{Case} \ref{casec}. As the first two cases cannot happen, we may assume that
			\[h^\star\left(A_{\frac{\partial f_i}{\partial x_{i,k}}}\right)=\min\limits_{j=1,\dots,m_i}\left\{h^\star\left(A_{\frac{\partial g_{i,j}}{\partial x_{i,k}}}\right)+v_{i,j}\right\}\]
			for all $i=1,\dots, N$.
			We consider two subcases,
			the case that $w_{i,k}\geq 0$ for each indices $i=1,\dots,N$ and $k\in \mathbf{I}_i^w$,
			and the case that there is $i\in {1,\dots,N}$ such that $w_{i,k}<0$ for some $k\in \mathbf{I}_i^w$. 
			
			First, we assume that $w_{i,k}\geq 0$ for each index $i$ and $k\in \mathbf{I}_i^w$. 
			Because we consider a dense polynomial $f_i$, its partial derivatives are also dense polynomials.
			Thus, we have $\boldsymbol{0}$ in the support $A_{p}$ for each $p\in \{\frac{\partial f_i}{\partial x_{i,k}}\mid{k\in \mathbf{I}_i^w}\}$.
			Therefore, we have $0\geq h^\star(A_{p})$.
			Since all $w_{i,k}\geq 0$, we also have $h^\star(A_{p})\geq 0$, and so we get $h^\star(A_p)=0$ for each $p$. 
			It further concludes that $w_{i,k}=0$ for all $i$ and $k\in \mathbf{I}_i^w$.
			Also, since \[h^\star\left(A_{\frac{\partial f_i}{\partial x_{i,k}}}\right)=\min\limits_{j=1,\dots,m_1}\left\{h^\star\left(A_{\frac{\partial g_{i,j}}{\partial x_{i,k}}}\right)+v_{i,j}\right\}=0\]
			for each $i$,
			we have $\min\limits_{j=1,\dots,m_i} v_{i,j}=0$. 
			We assume that 
			there is at least one index $i\in {1,\dots,N}$ such that 
			$v_{i,j}>0$ for some $j\in \{1,\dots, m_i\}$.
			Otherwise, $w$ can be considered as just a zero vector and there is nothing to prove.  
			Recall that $M_i^w$ is a subset of $\{1,\dots, m_i\}$ such that $v_l=\min\limits_{j=1,\dots, m_i} v_{i,j}=0$ for all $l\in M_i^w$.
			Then, we know that $M_i^w\subsetneq \{1,\dots, m_1\}$ for some $i=1,\dots,N$; otherwise, $F^w$ equals to $F$.
			Without loss of generality, let $i=1$ be such an index. 
			Then, the size of $M_1^w$ is exactly the number of variables $\lambda_{1,j}$ that appear in $F_1^w$ (i.e., $\lambda_{1,j}$ variable appears in $F_1^w$ if and only if $j\in M_1^w$).
			Without loss of generality, we further assume $M_1^w=\{1\ddd \hat{m}_1\}$ for some $\hat{m}_1<m_1$,
			and let $\widehat{\text{Jac}_1^w}$ be the submatrix of the left most $\hat{m}_1+1$ columns of $\text{Jac}_1^w$.
			If $(x^{\star},\lmd^{\star})$ is a nonzero solution to the facial system,
			then $\rank\,(\widehat{\text{Jac}_1^w}(x^{\star}))\le\hat{m}_1$.
			Define
			\[\widehat{W}_1^w=\{x\in \mathbb{C}^{\hat{n}}\mid\text{rank}(\widehat{\text{Jac}_1^w})\leq \hat{m}_1\}\]
			the determinantal variety of $\widehat{\text{Jac}_1^w}$.
			Then, by the similar argument from the proof for \cite[Theorem 3.1]{Nie2022degree}, we have 
			\[\cod\,(\widehat{W}_1^w\cap\mathbf{V}(g^w_{1,1}\ddd g^w_{1,\hat{m}_1}))\geq\hat{n}_1.\]
			Note that $g^w_{1,j}(x^{\star})=0$ for all $j=1\ddd m_i$. Since $\hat{m}_1<m_1$,
			for any $l$ such that $\hat{m}_1< l\le m_1$,
			the hypersurface $g^w_{1,l}$ intersects $\widehat{W}_1^w\cap\mathbf{V}(g^w_{1,1}\ddd g^w_{1,\hat{m}_1})$ properly from the genericity of $g^w_{1,l}$.
			It means that we have $\cod(\widehat{W}_1^w\cap U_1^w)>\hat{n}_1$, and hence $\widehat{W}_1^w\cap U_1^w\cap V_{-1}^w=\emptyset$. Therefore,
			such $x^{\star}$ does not exist.
			
			Lastly, we deal with the subcase that there exists $i\in \{1,\dots,N\}$ such that $w_{i,k}<0$ for some $k\in \mathbf{I}_i^w$. 
			Without loss of generality,
			assume that $i=1$ and suppose that $w_{1,\hat{k}}<0$ for some $\hat{k}\in \mathbf{I}_1^w$. 
			Since there is a negative entry $w_{1,\hat{k}}$, we have $h_1^\star(A_{g_{1,t}})<0$ for some $t\in \{1,\dots, m_1\}$.
			Furthermore, suppose that we have a root $(x^\star,\lambda^\star)$ of $F^w$. 
			Note that $g_{1,j}^w(x^\star)=0$ for all $j=1,\dots, m_i$.
			Let $t\in M_1^w$ be the index such that $h^\star(A_{g_{1,t}})<0$. 
			Then, by Lemma \ref{lm:geneuler}, we have
			\begin{align*}0=&\ h^\star(A_{g_{1,t}}) \lambda_{1,t}g^w_{1,t}(x^\star) =\sum_{i=1}^N\sum\limits_{k\in \mathbf{I}_i^w}w_{i,k}x_{i,k} \lambda_{1,t}\frac{\partial g_{1,t}^w}{\partial x_{i,k}}(x^\star)
				\\=&\ \sum_{k\in \mathbf{I}_1^w}w_{1,k}x_{1,k}\left(\frac{\partial f^w_1}{\partial x_{1,k}} (x^\star)-\sum_{j\in M_1^w\setminus\{t\}}\lambda_{1,j}\frac{\partial g^w_{1,j}}{\partial x_{1,k}} (x^\star)\right)
				\\&\ + \sum_{i=2}^N\sum\limits_{k\in \mathbf{I}_i^w}w_{i,k}x_{i,k} \lambda_{1,t}\frac{\partial g_{1,t}^w}{\partial x_{i,k}}(x^\star)
				\\=&\ h_{w_1}^\star\left(A_{f_1}\right)f^w_1 (x^\star)-\sum_{j\in M_1^w\setminus\{t\}}\lambda_{1,j}h_{w_1}^\star\left(A_{g_{1,j}}\right)g^w_{1,j} (x^\star)
				\\&\ + \sum_{i=2}^Nh_{w_i}^\star\left(A_{g_{1,t}}\right) \lambda_{1,t}g_{1,t}^w(x^\star).
			\end{align*}
			In the above, the third equality holds due to the fact that \[\frac{\partial f^w_1}{\partial x_{1,k}} (x^\star)-\sum_{j\in M_1^w}\lambda_{1,j}\frac{\partial g^w_{1,j}}{\partial x_{1,k}} (x^\star)=0,\]
			and the last equality is obtained by applying Lemma \ref{lm:geneuler}. 
			Let $$q=h_{w_1}^\star\left(A_{f_1}\right)f^w_1 -\sum_{j\in M_1^w\setminus\{t\}}\lambda_{1,j}h_{w_1}^\star\left(A_{g_{1,j}}\right)g^w_{1,j}
			+ \sum_{i=2}^Nh_{w_i}^\star\left(A_{g_{1,t}}\right) \lambda_{1,t}g_{1,t}^w$$
			be the polynomial obtained from the last equality. 
			We know that a point $x^\star$ lies in $\mathbf{V}(q)$. 
			It means that $q(x^\star)=h_{w_1}^\star(A_{f_1})f_1^w(x^\star)=0$ 
			since $x^\star\in U_1^w$. However, it contradicts the genericity of $f_1$.
		\end{proof}
		
		\begin{remark}
			\begin{enumerate}
				\item For the GNEP, if the defining functions are generic polynomials, then the set of complex KKT tuples is finite,
				and all KKT tuples lie in the torus when the GNEP only has equality constraints.
				This is implied by \cite[Theorem~3.1]{Nie2022degree}.
				In this case, Bernstein genericity implies that the mixed volume agrees with the algebraic degree.
				The explicit formula for the algebraic degree of generic GNEPs is studied in the recent paper \cite{Nie2022degree}.

				\item {Even when defining functions for the GNEP are not generic,
					the mixed volume still is an upper bound for the number of isolated solutions in the torus
					by Theorem~\ref{thm:BernsteinThm}.} 
				In this case, we may still find all complex KKT tuples using the homotopy continuation.
				However, it is still open in general that how to justify the completeness of solutions of a system found by the homotopy continuation. For partial results on the test of checking completeness, see \cite{leykin2018trace,brysiewicz2022sparse}.
			\end{enumerate}
		\end{remark}

		\section{Numerical examples}
		\label{sc:ne}
		In this section, we present some numerical experiments of 
		solving GNEPs of polynomials using the polyhedral homotopy continuation.
		We apply the software {\tt HomotopyContinuation.jl} to find complex KKT points of GNEPs by the polyhedral homotopy continuation,
		and apply {\tt Gloptipoly3} and {\tt SeDuMi} to implement the Moment-SOS hierarchy of semidefinite relaxations for verifying GNEs.
		The computation is executed in a {Macbook pro, 2 GHz Quad-Core Intel Core i5, 32 GB RAM}.
		
		When the GNEP is convex,
		if the complex KKT tuple $(x,\lambda_1 \ddd \lambda_N)$ satisfies
		\[ g_{i,j}(x)\ge0, \ \lambda_{i,j}\ge0 \quad  \mbox{for all}\quad i\in\{1,\dots,N\},\ j\in\mc{I}_i, \]
		i.e., $x$ is a KKT point, then $x$ is a GNE.
		For nonconvex GNEPs,
		the tuple of strategies $x$ is a GNE if and only if the
		\[ \dt\,:=\,\min_{i=1,\dots,N}\left\{ \min_{j\in\mc{I}} \{ g_{i,j}(x)\} ,\  \min_{j\in\mc{E}}\{ -|g_{i,j}(x)|\},\  \delta_i\right\}\ge 0 \]
		where $\delta_i$ is given by (\ref{eq:checking}).
		The $\delta$ is called the \emph{accuracy parameter} for $x$.
		In practical computation, one may not get $\dt\ge 0$ exactly, due to rounding-off errors.
		In this section, we regard $x$ being a GNE if $\dt\ge-10^{-6}$.

		\begin{example}
			\label{ep:NEP}
			\rm
			(i) Consider the $2$-player NEP in \cite{Nie2020nash}
			\begin{equation*}
				\mbox{1st player:} \left\{
				\begin{array}{cl}
					\min\limits_{x_1 \in \re^3} & \sum_{j=1}^3 x_{1,j}(x_{1,j}- j\cdot x_{2,j})\\
					s.t. & 1-x_{1,1}x_{1,2} \ge 0, \,1-x_{1,2}x_{1,3}\ge 0,\, x_{1,1} \ge 0,
				\end{array}\right.
			\end{equation*}
			\begin{equation*}
				\mbox{2nd player:}  \left\{
				\begin{array}{cl}
					\min\limits_{x_2 \in \re^3} & %%x_{1,1}x_{1,2}x_{1,3}+ x_{2,1}x_{2,2}x_{2,3}+
					\prod_{j=1}^3 x_{2,j} +
					\sum\limits_{\mathclap{\substack{1\le i<j\le 3\\1\le k\le 3}}} x_{1,i}x_{1,j}x_{2,k}+
					\sum\limits_{\mathclap{\substack{1\le i\le 3\\1\le j<k\le 3}}} x_{1,i}x_{2,j}x_{2,k}
					\\
					s.t. & 1-(x_{2,1})^2 - (x_{2,2})^2 - (x_{2,3})^2 = 0.
				\end{array}\right.
			\end{equation*}
			This is a nonconvex NEP since both players' optimization problems are nonconvex.
			Moreover, the feasible set for the first player's optimization problem is unbounded.
			By implementing the polyhedral homotopy continuation on the complex KKT system,
			we got $252$ complex KKT tuples,
			and $8$ of them satisfy the KKT system (\ref{eq:KKTsystem}).
			Since this is a nonconvex problem,
			we ran Algorithm~\ref{ag:selectGNE} for selecting NEs.
			We obtained four NEs $u=(u_1,u_2)$ with
			\[
			\begin{array}{ll}
				u_1=(0.3198,  0.6396, -0.6396 ), & u_2=( 0.6396,    0.6396,    -0.4264); \\
				u_1=(0.0000, 0.3895, 0.5842 ),   & u_2=(-0.8346, 0.3895, 0.3895); \\
				u_1=(0.2934, -0.5578, 0.8803 ),  & u_2=( 0.5869, -0.5578,  0.5869);\\
				u_1=(0.0000, -0.5774, -0.8660 ), & u_2=( -0.5774, -0.5774, -0.5774). \\
			\end{array}
			\]
			Their accuracy parameters are respectively
			\[
			-5.2324\cdot 10^{-11}, \, -1.7619\cdot 10^{-9}, \, -4.8633\cdot 10^{-9},
			-7.1933\cdot 10^{-9}.
			\]
			Note that for this NEP,
			the mixed volume of the complex KKT system equals $252$.
			The polyhedral homotopy found all complex KKT tuples,
			so all NEs are obtained by our method.
			It took about $7.81$ seconds to find all NEs, including $4$ seconds to find all complex KKT tuples, and about $3.81$ seconds to verify NEs.
			
			(ii) If the second player's objective becomes
			\[
			-\prod_{j=1}^3 x_{2,j} +
			\sum\limits_{\mathclap{\substack{1\le i\le 3\\1\le j<k\le 3}}} x_{1,i}x_{2,j}x_{2,k}
			-\sum\limits_{\mathclap{\substack{1\le i<j\le 3\\1\le k\le 3}}} x_{1,i}x_{1,j}x_{2,k},
			\]
			then the polyhedral homotopy continuation found $252$ complex KKT tuples,
			and there are $3$ of them satisfying the KKT system (\ref{eq:KKTsystem}).
			However, none of these KKT points are NEs,
			by Algorithm~\ref{ag:selectGNE}.
			Indeed, since the mixed volume for the complex KKT system equals $252$,
			all complex KKT tuples were found by homotopy continuation.
			Therefore, we detected that this NEP does not have any NE.
			It took around $3$ seconds to solve the complex KKT system,
			{and $1.09$ seconds to detect the nonexistence of NEs}.
		\end{example}
		
		\begin{example}
			\label{ep:GNEPvariation}
			\rm
			Consider a GNEP variation of the problem in Example~\ref{ep:NEP}(i).
			\begin{equation*}
				\mbox{1st player:} \left\{
				\begin{array}{cl}
					\min\limits_{x_1 \in \re^3} & \sum_{j=1}^3 x_{1,j}(x_{1,j}- j\cdot x_{2,j})\\
					s.t. & x_{2,3}-x_{1,1}x_{1,2} \ge 0, \,x_{2,1}-x_{1,2}x_{1,3}\ge 0,\, x_{1,1}-x_{2,2} \ge 0,
				\end{array}\right.
			\end{equation*}
			\begin{equation*}
				\mbox{2nd player:}  \left\{
				\begin{array}{cl}
					\min\limits_{x_2 \in \re^3} & %%x_{1,1}x_{1,2}x_{1,3}+ x_{2,1}x_{2,2}x_{2,3}+
					\prod_{j=1}^3 x_{2,j} +
					\sum\limits_{\mathclap{\substack{1\le i<j\le 3\\1\le k\le 3}}} x_{1,i}x_{1,j}x_{2,k}+
					\sum\limits_{\mathclap{\substack{1\le i\le 3\\1\le j<k\le 3}}} x_{1,i}x_{2,j}x_{2,k}
					\\
					s.t. & 1-(x_{1,1}x_{2,1})^2 - (x_{2,2})^2 - (x_{2,3})^2 = 0.
				\end{array}\right.
			\end{equation*}
			Similar to the problem in Example~\ref{ep:NEP}(i),
			this is a nonconvex GNEP, and the first player has an unbounded feasible strategy set.
			By implementing the polyhedral homotopy continuation on the complex KKT system,
			we computed the mixed volume $512$ and found $484$ complex KKT tuples,
			and $11$ of them satisfy the KKT system (\ref{eq:KKTsystem}).
			Since this is a nonconvex problem,
			we ran Algorithm~\ref{ag:selectGNE} for selecting GNEs.
			We obtained two GNEs $u=(u_1,u_2)$ with
			\[
			\begin{array}{ll}
				u_1=(0.8188, -0.3213, -0.3947),  & u_2=(0.8868, 0.6353, -0.2631);\\
				u_1=(0.5873, -0.5993,  0.6091),  & u_2=(1.1747, -0.5993, 0.4061). \\
			\end{array}
			\]
			Their accuracy parameters are respectively
			\[
			-4.0433\cdot 10^{-9}, \quad  -6.7675\cdot 10^{-9}.
			\]
			{It took about $9.85$ seconds to find all GNEs including $4$ seconds to solve the complex KKT system, and about $5.85$ seconds to verify GNEs.}
		\end{example}

		\begin{example}\rm
			\label{ep:convexGNEP}
			Consider the 2-player convex GNEP in \cite{Nie2021convex}
			\[
			\begin{array}{lllll}
				%\label{eq:rational_example}
				\min\limits_{x_1\in\re^2} & \sum\limits_{j=1}^2(x_{1,j}-1)^2+x_2(x_{1,1}-x_{1,2}) & \vline & \min\limits_{x_2\in\re^1} & (x_2)^3-x_{1,1}x_{1,2}x_2-x_2\\
				\st              & 2-x_1^{\top}x_1-x_2\ge0; & \vline & \st              & 3x_2-x_1^{\top}x_1\ge0 ,\,1-x_2\ge0.
			\end{array}
			\] 
			By implementing the polyhedral homotopy continuation on the complex KKT system,
			we knew the mixed volume is $23$, and we got $17$ complex KKT tuples.
			For these KKT tuples, only one of them satisfies the KKT system (\ref{eq:KKTsystem}).
			Because this is a convex GNEP,
			we got a GNE $u:=(u_1,u_2)$ from this KKT tuple with
			\[u_1=(0.4897, 1.0259),\quad u_2=(0.7077).\]
			It took around $2$ seconds to solve the complex KKT system.
		\end{example}
		
		\begin{example}\rm
			\label{ep:nonconvexGNEP}
			Consider a 2-player GNEP
			\begin{equation*}
				\mbox{1st player:} \left\{
				\begin{array}{cl}
					\min\limits_{x_1 \in \re^2} & 3x_{2,1}(x_{1,1})^3+5(x_{1,2})^3 -2\sum\nolimits_{j=1}^2x_{1,j}\cdot \sum\nolimits_{j=1}^2x_{2,j}\\
					s.t. & 5x_{1,1}-2x_{1,2}+3x_{2,2}-1\ge0, \ 3-x_{2,1}\cdot x_1^{\top}x_1\ge0,\\
					&  x_{1,1}\ge-2,\ x_{1,2}\ge1;
				\end{array}\right.
			\end{equation*}
			\begin{equation*}
				\mbox{2nd player:}  \left\{
				\begin{array}{cl}
					\min\limits_{x_2 \in \re^2} & (2x_{1,1}+3x_{1,2})(x_{2,1})^3-3x_{2,1}+7(x_{2,2})^2+5x_{1,1}x_{1,2}x_{2,2}
					\\
					s.t. & 7x_{1,2}+3x_{2,2}-5x_{2,1}^2+3\ge0, \ 2x_{2,1}\ge-1,\\
					& 2-x_{2,2}\ge0,\ 5+x_{2,2}\ge0.
				\end{array}\right.
			\end{equation*}
			This is a nonconvex GNEP.
			By implementing the polyhedral homotopy continuation on the complex KKT system,
			we knew the mixed volume is $480$ and polyhedral continuation found exactly $480$ complex KKT tuples.
			We ran Algorithm~\ref{ag:selectGNE}
			and obtained the unique GNE $u:=(u_1,u_2)$ with
			\[u_1=(0.7636,1.0000),\quad u_2=(0.4700,-0.2727),\quad \delta=-1.0220\cdot10^{-8}.\]
			Note that for this GNEP,
			the mixed volume of the complex KKT system coincides with the number of complex KKT tuples we found.
			The polyhedral homotopy found all complex KKT tuples,
			so all GNEs are obtained by our method.
			It took around $5.75$ seconds to find all GNEs including $4$ seconds to solve the complex KKT system,
			{and $1.75$ seconds to select the GNE.}
		\end{example}
		
		\begin{example} \rm  \label{ep58}
			Consider a GNEP whose optimization problems are
			\begin{equation*}
				\mbox{1st player:} \left\{
				\begin{array}{cl}
					\min\limits_{x_1 \in \re^2} & 2(x_{1,1})^2+7(x_{1,2})^2+3x_{1,1}+5x_{1,2}\\
					s.t. & 1-2(x_{1,1})^2-(x_{1,2})^2-3(x_{2,1})^2-5(x_{2,2})^2\ge0,\\
					&  1-x_{1,1}\ge0,\ \frac{1}{2}-x_{1,2}\ge0;
				\end{array}\right.
			\end{equation*}
			\begin{equation*}
				\mbox{2nd player:}  \left\{
				\begin{array}{cl}
					\min\limits_{x_2 \in \re^2} & 3(x_{2,2})^2-4x_{2,1}x_{2,2}\\
					s.t. & 3(x_{1,1})^2+(x_{1,2})^2+\frac{7}{10}(x_{2,1})^2+6(x_{2,2})^2-1\ge0,\\
					& 7-x_{2,1}\ge0,\ x_{2,2}-\frac{3}{10}\ge0,\  \frac{8}{10}-x_{2,2}\ge0.
				\end{array}\right.
			\end{equation*}
			This is a nonconvex GNEP.
			By implementing the polyhedral homotopy continuation on the complex KKT system,
			we computed the mixed volume is $168$ and polyhedral homotopy found $168$ complex KKT tuples.
			However, none of them are GNEs.
			It took around $3$ seconds to solve this problem, including $3$ seconds to solve the complex KKT system,
			{and $0.001$ seconds to detect the nonexistence of GNEs.}
		\end{example}
		
		\begin{example}\rm
			\label{ep:A3}
			Consider a convex GNEP of $3$ players.
			For $i=1,2,3$, the $i$th player aims to minimize the quadratic function
			$$f_i(x)=\frac{1}{2}x_i^{\top}A_ix_i+x_i^{\top}(B_i\xmi+b_i).$$
			All variables have box constraints $-10\le x_{i,j}\le 10$, for all $i,j$.
			In addition to them, the first player has linear constraints $x_{1,1}+x_{1,2}+x_{1,3}\le 20,\,x_{1,1}+x_{1,2}-x_{1,3}\le x_{2,1}-x_{3,2}+5$;
			the second player has $x_{2,1}-x_{2,2}\le x_{1,2}+x_{1,3}-x_{3,1}+7$;
			and the third player has $x_{3,2}\le x_{1,1}+x_{1,3}-x_{2,1}+4.$
			
			(i) Consider the case that the values of parameters are set as in \cite[Example A.3]{facchinei2010penalty}:
			\[
			\begin{array}{cccc}
				A_1=\lvt\begin{array}{ccc}
					20&5&3\\5&5&-5\\3&-5&15
				\end{array}\rvt,\ 
				A_2=\lvt\begin{array}{ccc}
					11&-1\\-1&9
				\end{array}\rvt,\ 
				A_3=\lvt\begin{array}{ccc}
					48&39\\39&53
				\end{array}\rvt,\\
				B_1=\lvt\begin{array}{cccc}
					-6&10&11&20\\10&-4&-17&9\\15&8&-22&21
				\end{array}\rvt,\ 
				B_2=\lvt\begin{array}{ccccc}
					20&1&-3&12&1\\10&-4&8&16&21
				\end{array}\rvt,\\
				B_3=\lvt\begin{array}{ccccc}
					10&-2&22&12&  16  \\
					9&19&21&-4&20
				\end{array}\rvt,\ 
				b_1=\lvt\begin{array}{ccc}
					1\\-1\\1
				\end{array}\rvt,\ 
				b_2=\lvt\begin{array}{ccc}
					1\\0
				\end{array}\rvt,\ 
				b_3=\lvt\begin{array}{ccc}
					-1\\2
				\end{array}\rvt.
			\end{array}\]
			This is a convex GNEP since for all $i\in\{1,2,3\}$, the $A_i$ is positive semidefinite and all constraints are linear.
			By implementing the polyhedral homotopy continuation on the complex KKT system,
			we got the mixed volume $12096$, and polyhedral homotopy found $11631$ complex KKT tuples.
			There are $5$ GNEs obtained by Algorithm~\ref{ag:selectGNE},
			which are presented in the following table.
			\begin{table}[htb]
				\renewcommand{\arraystretch}{1}
				\centering
				\renewcommand{\arraystretch}{1.20}
				\begin{tabular}{c||c|c|c}  
					& $u_1$ &  $u_2$ & $u_3$\\ \hline\hline
					1& (-0.3805,-0.1227,-0.9932) &  (0.3903,1.1638)   & (0.0504,0.0176)   \\\hline
					2&(-0.9018,-4.4017,-2.1791) & (-2.0034,-2.4541)  & (-0.0316,2.9225)  \\\hline
					3& (-0.8039,-0.3062,-2.3541) & (0.9701, 3.1228)  & (0.0751,-0.1281)  \\\hline
					4& (1.9630,-1.3944, 5.1888) & (-3.1329,-10.0000) & (-0.0398,1.6392)  \\\hline
					5& (0.6269,10.0000,9.3731)  &  (1.8689,10.0000)  & (0.3353,-10.0000) \\
				\end{tabular}
			\end{table}
			
			\noindent It took around 177 seconds to solve the complex KKT system.
			We would like to remark that in \cite{facchinei2010penalty} and \cite{Nie2021convex}, only the first GNE was found,
			and the second to the fourth GNEs are new solutions found by our algorithm.

			(ii) If we let
			\[
			A_1=\lvt\begin{array}{ccc}
				1&2&3\\2&5&-5\\3&-5&15
			\end{array}\rvt,\
			\]
			and all other parameters be given as in (i),
			then this GNEP is nonconvex.
			By implementing the polyhedral homotopy continuation on the complex KKT system,
			the mixed volume equals $12096$ and we got $11620$ complex KKT tuples,
			and five of them satisfy the KKT condition (\ref{eq:KKTsystem}).
			Since this is a nonconvex problem,
			we ran Algorithm~\ref{ag:selectGNE} for selecting GNEs,
			and obtained one GNE $u=(u_1,u_2,u_3)$ with
			\[u_1=(0.9968,10.0000,9.0032),\ 
			u_2=(0.6668,10.0000),\ 
			u_3=(0.7283,-10.0000).\]
			The accuracy parameter is $-9.5445\cdot10^{-7}$.
			It took around $209.68$ seconds to find all GNEs including $207$ seconds to solve the complex KKT system,
			{and $2.68$ seconds to select the GNE.}
		\end{example}
		
		\subsection{Comparison with existing methods}
		In this subsection, we compare the performance of Algorithm~\ref{ag:selectGNE} with some existing methods for solving GNEPs,
		such as the augmented Lagrangian method (ALM) in \cite{kanzow2016augmented},
		Gauss-Seidel method (GSM) in \cite{nie2021gauss},
		the interior point method (IPM) in \cite{dreves2011solution},
		and the semidefinite relaxation method (KKT-SDP) in \cite{Nie2021convex}.
		We tested these methods on all GNEPs of polynomials in Examples~\ref{ep:NEP}-\ref{ep:A3}.
		
		Given a computed tuple $u=(u_1,\dots, u_N)$ for an $N$-player game.
		Then, $u$ is a GNE if and only if $\delta\ge0$.
		For the KKT-SDP method, 
		we say the method finds a GNE successfully whenever $\dt\ge-10^{-6}$ since $\dt\geq 0$ may not be possible due to a numerical error.
		For other earlier algorithms mentioned above, since they are iterative methods, the stopping criterion is given as the following:
		For the computed tuple $u$, when $\min\limits_{i=1,\dots,N}\left\{ \min\limits_{j\in\mc{I}} \{ g_{i,j}(x)\} ,\  \min\limits_{j\in\mc{E}}\{ -|g_{i,j}(x)|\}\right\}\ge -10^{-6}$,
		we solve (\ref{eq:checking}) for each $i$.
		If we further have $\dt\ge-10^{-6}$, then we stop the iteration and report that the method found a GNE successfully.

		For the ALM, GSM and IPM, the same parameters are applied as in \cite{dreves2011solution,kanzow2016augmented,nie2021gauss}.
		In the augmented Lagrangian method, full penalization is used,
		and a Levenberg-Marquardt type method (see \cite[Algorithm~24]{kanzow2016augmented})
		is implemented to solve penalized subproblems.
		For the Gauss-Seidel method,
		normalization parameters are updated as (4.3) in \cite{nie2021gauss},
		and Moment-SOS relaxations are used to
		solve normalized subproblems.
		We let $1000$ be the maximum number of iterations for the ALM and IPM,
		and at most $100$ iterations are allowed in the GSM.
		For initial points,
		we use $(0,0,0,\frac{1}{\sqrt{3}},\frac{1}{\sqrt{3}},\frac{1}{\sqrt{3}})$ for Examples~\ref{ep:NEP}(i-ii),
		$(0,0,0,0,-\frac{1}{\sqrt{2}},\frac{1}{\sqrt{2}})$ for Example~\ref{ep:GNEPvariation},
		$(0, 0, 1, 1)$ for Example~\ref{ep:nonconvexGNEP},
		$(0, 0, 0, \frac{1}{\sqrt{5}})$ for Example~\ref{ep58},
		and the zero vector for all other problems.
		For the one-shot KKT-SDP method, randomly generated positive semidefinite matrices are exploited to formulate polynomial optimization.
		Note that the ALM, IPM and KKT-SDP are designed for finding a KKT point of the GNEP.
		When the GNEP is convex (e.g., Examples~\ref{ep:convexGNEP} and \ref{ep:A3}(i)),
		the limit point is guaranteed to be a GNE, if these methods produce a convergent sequence.
		However, since we in general do not make any convexity assumption,
		it is possible that these methods converge to a KKT point which is not a GNE.
		For the ALM and IPM, the produced sequence is
		considered convergent to a KKT point if the last iterate satisfies the KKT
		conditions up to a small round-off error (say, $10^{-6}$).
		If the iterations are convergent
		but the stopping criterion is not met,
		we still solve (\ref{eq:checking}) to check if
		the latest iterating point is a GNE or not.
		
		The numerical results are shown in Table~\ref{tab:comparison}.
		In the table, ``time'' is the time consumption in seconds for solving the problem,
		and ``error'' is the quantity $-\delta$ at the computed solution.
		In general, the method can be regarded to solve the GNEP successfully if the error is small (e.g., less than $10^{-6}$).
		For Algorithm~\ref{ag:selectGNE},
		when there are more than one GNEs obtained,
		we present the largest error among these GNEs.
		\begin{table}[ht]
			\renewcommand{\arraystretch}{1.20}
			\begin{tabular}{c|c||c|c|c|c|c}  
				\multicolumn{2}{c||}{Example}           & ALM    & IPM & GSM & KKT-SDP & Algorithm~\ref{ag:selectGNE} \\ \hline\hline
				\multirow{2}{*}{\ref{ep:NEP}(i)} & time  & 27.52 & 13.34 & \multirow{2}{*}{Fail}  & 2.95  & 12.6 \\ \cline{2-4}\cline{6-7}
				& error & 4.67 & 4.67 & & 1.48 & $< 8\cdot 10^{-9}$ (4 GNEs) \\ \hline
				
				\multirow{2}{*}{\ref{ep:NEP}(ii)}
				& time & 32.03 & 8.04 & \multirow{2}{*}{Fail}  & 2.92  & 7.8 \\ \cline{2-4}\cline{6-7}
				& error & 1.11 & 1.11 & & 0.19 & no GNE \\ \hline
				
				\multirow{2}{*}{\ref{ep:GNEPvariation}}
				& time  & \multirow{2}{*}{Fail} & \multirow{2}{*}{Fail} & \multirow{2}{*}{Fail} & 4.65  & 9.9 \\ \cline{2-2}\cline{6-7}
				& error &  &  & & 0.66 & $< 7\cdot 10^{-9}$ (2 GNEs) \\ \hline
				
				\multirow{2}{*}{\ref{ep:convexGNEP}}
				& time  & 0.72 & 3.14 & 4.45 & 1.51 & 2.0 \\ \cline{2-7}
				& error & $2\cdot 10^{-7}$ & $2\cdot 10^{-7}$ & $2\cdot 10^{-7}$ & $8\cdot 10^{-9}$ & $1\cdot 10^{-8}$ \\ \hline
				
				\multirow{2}{*}{\ref{ep:nonconvexGNEP}}
				& time  & \multirow{2}{*}{Fail} & 1.69 & 11.47 & 17.89 & 5.75  \\ \cline{2-2}\cline{4-7}
				& error &  & $2\cdot 10^{-7}$ & $4\cdot 10^{-7}$ & $1\cdot 10^{-6}$  & $2\cdot 10^{-8}$ \\ \hline
				
				\multirow{2}{*}{\ref{ep58}}
				& time  & \multirow{2}{*}{Fail} & \multirow{2}{*}{Fail} & \multirow{2}{*}{Fail} & 1.51 & 3 \\ \cline{6-7}
				& error &  &  &  & no GNE & no GNE \\ \hline
				
				\multirow{2}{*}{\ref{ep:A3}(i)}
				& time  & 1.50 & 3.12 & \multirow{2}{*}{Fail} & 11.55 & 177 \\ \cline{2-4}\cline{6-7}
				& error & $1\cdot 10^{-7}$ & $2\cdot 10^{-7}$  &  & $2\cdot 10^{-7}$ & $< 1\cdot 10^{-6}$ (5 GNEs) \\ \hline
				
				\multirow{2}{*}{\ref{ep:A3}(ii)}
				& time  & 59.93 & 16.19 & \multirow{2}{*}{Fail} & 11.29 & 210 \\ \cline{2-4}\cline{6-7}
				& error & $123.22$ & 123.22 &  & 123.22& $1\cdot 10^{-7}$ \\
			\end{tabular}%
			\smallskip
			\caption{Comparison with other methods. The ``time" gives the consumed time (in seconds) for finding a GNE or a KKT point,
				and the ``error" measures the quantity $-\delta$ of the computed GNE candidate.}\label{tab:comparison}

		\end{table}
		
		The comparison is summarized as follows:
		\begin{enumerate}
			\item The augmented Lagrangian method converges to a KKT point that is not a GNE for Examples~\ref{ep:NEP}(i-ii) and \ref{ep:A3}(ii).
			For Example~\ref{ep:GNEPvariation}, the iteration cannot proceed because the maximum penalty parameter was reached at the $14$th iteration.
			For Examples~\ref{ep:nonconvexGNEP} and \ref{ep58}, it fails to converge because the penalized subproblem cannot be solved accurately.
			
			\item The interior point method converges to a KKT point that is not a GNE for Examples~\ref{ep:NEP}(ii), \ref{ep:nonconvexGNEP} and \ref{ep:A3}(ii).
			For Examples~\ref{ep:GNEPvariation} and \ref{ep58}, the algorithm does not converge. In this problem, the Newton-type directions usually do not satisfy sufficient descent conditions.
			\item For Examples~\ref{ep:NEP}(i-ii) and \ref{ep:A3}(i), the Gauss-Seidel method failed to converge and it alternated between several points.
			For Examples~\ref{ep:GNEPvariation}, \ref{ep58} and \ref{ep:A3}(ii), the iteration cannot proceed at some stages since global minimizers for normalized subproblems cannot be obtained. Usually, this is because the normalized subproblem is infeasible or unbounded.
			\item The semidefinite relaxation method obtained a KKT point that is not a GNE for Examples~\ref{ep:NEP}(i-ii), \ref{ep:GNEPvariation} and \ref{ep:A3}(ii).
			
			\item Algorithm~\ref{ag:selectGNE} detected nonexistence of GNEs for Examples~\ref{ep:NEP}(ii) and \ref{ep58}.
			We would like to remark that if there exist KKT points that are not GNEs, then the semidefinite relaxation method may not detect the nonexistence of GNEs.
			For all other GNEPs,
			Algorithm~\ref{ag:selectGNE} found at least one GNE.
			Moreover, for Examples~\ref{ep:NEP}(i), \ref{ep:GNEPvariation} and \ref{ep:A3},
			Algorithm~\ref{ag:selectGNE} found more than one GNE,
			and the completeness of GNEs are guaranteed for Examples~\ref{ep:NEP}(i), \ref{ep:GNEPvariation} and \ref{ep:nonconvexGNEP}.

		\end{enumerate}

		\subsection{GNEPs of polynomials with randomly generated coefficients}
		We present numerical results of Algorithm~\ref{ag:selectGNE} on GNEPs defined by polynomials whose coefficients are randomly generated.
		For the GNEP with $N$ players, we assume that all players have the same dimension for their strategy vectors, i.e., $n_1=n_2=\cdots =n_N$. The $i$th player's optimization problem is given by 
		\be
		\label{eq:randgnep}
		\left\{ \begin{array}{cl}
			\min\limits_{\xpi\in \Rni}  &  \fpi(x_i,x_{-i}) \\
			\st & -x_i^{\top}A_ix_i + \xmi^{\top}B_ix_i + c_i^{\top}x \ge d_i.
		\end{array} \right.
		\ee
		In the above, we have $A_i=R_i^{\top} R_i$ with a randomly generated matrix $R_i\in\re^{n_i\times n_i}$. Also, $B_i\in\re^{{n_i}\times (n-n_i)}$, $c_i\in\re^n$, $d_i\in\re$ are randomly generated real matrices or vectors.
		Under this setting, the constraining function of (\ref{eq:randgnep}) is a convex polynomial in $x_i$, and the $X_i(\xmi)$ is compact, for all $\xmi\in\re^{n-n_i}$.
		
		For the objective function $f_i$,
		we consider two cases.
		First, we let
		\[f_i:=x_i^{\top}\Sigma_ix_i + \xmi^{\top}\Lambda_ix_i + c_i^{\top}x,\]
		where $\Sigma_i=\Theta_i^{\top}\Theta_i$ with a randomly generated matrix $\Theta_i\in\re^{n_i\times n_i}$,
		and $\Lambda_i$ (resp., $c_i$) is a randomly generated matrix (resp., vector) in $\re^{{n_i}\times (n-n_i)}$ (resp., in $\re^n$).
		In this case, the GNEP given by (\ref{eq:randgnep}) is convex, and all KKT points are GNEs.
		The second case is for GNEPs without convexity settings. 
		We consider a degree $d$ dense polynomial with randomly generated real coefficients,
		i.e.,
		\[f_i:=\zeta^{\top}[x]_d,\]
		where $\zeta$ is a randomly generated real vector of the proper size.
		To choose real matrices, vectors and coefficients randomly, we use the {\tt Matlab} function {\tt unifrnd} that generates real numbers following the uniform distribution.
		
		The numerical results are presented in Table \ref{tab:randgnep}.
		By Theorem \ref{tm:mvfound},
		if the mixed volume many complex KKT points are obtained,
		then Algorithm \ref{ag:selectGNE} can find all GNEs or detect the nonexistence of GNEs. 
		Since we consider random examples, the homotopy method mostly finds all mixed volume many KKT points. 
		As the problem sizes grow, there are some cases where the homotopy method cannot find all mixed volume many KKT points, due to numerical issues. 
		
		\section{Conclusions and discussions}
		This paper studies a new approach for solving GNEPs of polynomials using the polyhedral homotopy continuation and the Moment-SOS relaxations.
		We show that under some generic assumptions,
		the mixed volume and the algebraic degree for the complex KKT system are identical,
		and our method can find all GNEs or detect the nonexistence of GNEs.
		Some numerical experiments are presented to show the effectiveness of our method.
		
		For future work,
		it is interesting to find local GNEs,
		i.e.,  find $x=(x_1\ddd x_N)$ such that each $x_i$ is local minimizer for $\mbF_i(\xmi)$.
		Note that every local GNE satisfies the KKT condition.
		However, it is difficult to select local GNEs from KKT points,
		especially when \emph{the second-order sufficient optimality conditions} (see \cite{Brks}) are not satisfied.
		Moreover, when the $|\mc{K}_{\cpx}|$ is strictly less than the mixed volume for (\ref{eq:KKTep}),
		how do we know whether our method finds all GNEs or detects nonexistence of GNEs or not?
		These questions are mostly open, to the best of the authors' knowledge.
		
		\begin{table}[ht]
			\renewcommand{\arraystretch}{1.20}
			
			\begin{tabular}{c|c|c|c|c}  
				\multicolumn{5}{@{}l}{(a) Degree $2$ convex GNEPs}\\
				$N$ & $n_i$ & mixed volume & rate of success & average time \\[.3pt] \hline\hline 
				2 & 2 & 25 &      100\%           &  0.0575  \\ 
				2 & 3 &49 &      100\%           &  0.1721  \\ 
				2 & 4 &  81 &    100\%           & 0.9539 \\     
				3 & 2 &  125  &  100\%           & 0.9118 \\ 
				3 & 3 & 343   &  100\%             &3.4150  \\     
			\end{tabular}
			\bigskip
			
			\begin{tabular}{c|c|c|c|c|c}  
				\multicolumn{6}{@{}l}{(b) Degree $d$ nonconvex GNEPs}\\
				
				$d$ & $N$ & $n_i$ & mixed volume & rate of success & average time (seconds)\\[.3pt] \hline\hline 
				\multirow{5}{*}{2}   & 2 & 2 & 25 &     100\%            &  0.0563 + 1.1330\\  
				& 2 & 3 &  49 &    100\%            &  0.1802 + 1.5098\\  
				& 2 & 4 &   81 &   100\%            & 0.8819 + 1.9762 \\  
				& 3 & 2 &  125  &  100\%            &  0.8473 + 3.1890\\  
				& 3 & 3 &  343  &  100\%            &  3.3804 + 6.9738\\  \hline
				\multirow{4}{*}{3} & 2 & 2 &100          &   100\%    & 0.1893 + 2.5667  \\ 
				& 2 & 3 &  484  & 100\%            &  2.18 + 5.7500 \\ 
				& 2 & 4 &   2116   &    98\%       & 21.483 + 17.3477 \\      & 3 & 2 &  1000  &   97\%          & 5.255 + 14.4360  \\ \hline
				\multirow{3}{*}{4}    & 2 & 2 &   289     & 100\%        & 0.8270 +4.4256 \\ 
				& 2 & 3 &  2809   & 95\%           & 24.533 + 21.9054 \\ 
				& 3 & 2 &  4913    &     95\%      & 44.0899 + 40.6792  \\

			\end{tabular}
			\bigskip
			\caption{Numerical results for random GNEPs. $N$ and $n$ indicate the number of players and the number of variables for each player respectively. The ``rate of success'' indicates the percentage of GNEPs such that the homotopy continuation finds the mixed volume many KKT points. The ``average time'' represents the average of the elapsed times for $\text{(KKT points computation)}+\text{(GNE selection)}$ in seconds. For convex problems, the elapsed time is only measured for the KKT points computation.}\label{tab:randgnep}

		\end{table}
		
		\section*{Acknowledgements}
		The authors would like to thank Jiawang Nie for motivating this paper and helpful comments.
		Xindong Tang is partially supported by the Start-up Fund P0038976/BD7L from The Hong Kong Polytechnic University.

	\end{document}